\documentclass{article}

\usepackage{arxiv}

\usepackage{graphicx}%
\usepackage{multirow}%
\usepackage{amsmath,amssymb,amsfonts}%
\usepackage{amsthm}%
\usepackage{mathrsfs}%
\usepackage[title]{appendix}%
\usepackage[dvipsnames]{xcolor}%
\usepackage{textcomp}%
\usepackage{manyfoot}%
\usepackage{booktabs}%
\usepackage{algorithm}%
\usepackage{algorithmicx}%
\usepackage{algpseudocode}%
\usepackage{listings}%
\usepackage{subcaption}
\usepackage{enumitem}
\usepackage{hyperref}
\usepackage{tikz}
\usepackage{comment}
\usepackage{adjustbox}

\newcommand{\EM}[1]{\textcolor{Green}{#1}}

\newcommand{\R}{\mathbb{R}}

\newcommand{\X}{\mathcal{X}}

\newcommand{\y}{\boldsymbol{y}}
\newcommand{\x}{\boldsymbol{x}}
\newcommand{\z}{\boldsymbol{z}}
\newcommand{\e}{\boldsymbol{e}}
\newcommand{\w}{\boldsymbol{w}}

\newcommand{\K}{\boldsymbol{K}}
\newcommand{\D}{\boldsymbol{D}}
\newcommand{\M}{\boldsymbol{M}}
\newcommand{\I}{\boldsymbol{I}}

\newtheorem{theorem}{Theorem}
\newtheorem{proposition}{Proposition}
\newtheorem{lemma}{Lemma}
\newtheorem{corollary}{Corollary}

\numberwithin{equation}{section}

\begin{document}
\title{Adaptive Weighted Total Variation boosted by learning techniques in few-view tomographic imaging}

\author{Elena Morotti \\
Department of Political and Social Sciences, \\ University of Bologna,\\ Bologna, 40126, Italy. \\ \texttt{elena.morotti4@unibo.it}. \\
\And Davide Evangelista \\
Department of Computer Science and Engineering, \\ University of Bologna,\\ Bologna, 40126, Italy. \\
\texttt{davide.evangelista5@unibo.it} \\
\And Andrea Sebastiani \\
Department of Physics, Informatics and Mathematics, \\ University of
Modena and Reggio Emilia, \\ Modena, 41125, Italy. \\ \texttt{asebastiani@unimore.it}. \\
\And Elena Loli Piccolomini \\
Department of Computer Science and Engineering, \\ University of Bologna,\\ Bologna, 40126, Italy. \\ \texttt{elena.loli@unibo.it} \\}

\maketitle

\begin{abstract}
This study presents the development of a spatially adaptive weighting strategy for Total Variation regularization, aimed at addressing under-determined linear inverse problems. The method leverages the rapid computation of an accurate approximation of the true image (or its gradient magnitude) through a neural network. Our approach operates without requiring prior knowledge of the noise intensity in the data and avoids the iterative recomputation of weights.\\
Additionally, the paper includes a theoretical analysis of the proposed method, establishing its validity as a regularization approach.  
This framework integrates advanced neural network capabilities within a regularization context, thereby making the results of the networks interpretable. The results are promising as they enable high-quality reconstructions from limited-view tomographic measurements.
\end{abstract}
\vspace{2pc}


\section{Introduction}

In this work, we investigate an imaging application aimed at reconstructing medical images from a limited number of tomographic measurements. This application is of particular interest due to the need to minimize radiation exposure, which poses health risks to patients, as well as to reduce data acquisition times.  \\
In a discrete framework, this problem can be formulated as a linear inverse problem of the form:
	\begin{equation}\label{eq:inverse_problem}
		\y^\delta = \K \x^{GT} + \e, \quad || \e ||_2 \leq \delta,
	\end{equation}
where $\x^{GT} \in \X \subseteq \R^n$ is the (unknown) ground truth image (reshaped into a one-dimensional array), $\K \in \R^{m \times n}$ is an under-sampled, i.e.  $m \leq n$, linear projecting operator,  $\y^\delta \in \R^m$ is the acquired noisy sinogram, and $\e \in \R^m$ is the noise, whose norm is bounded by $\delta$. 
Note that, since $m \leq n$, the kernel of $\K$ is generally non-trivial, which implies that \eqref{eq:inverse_problem} admits infinite solutions.  Our investigation is specifically focused on the context of X-ray fan beam Computed Tomography (CT). In this setting, the operator $\K$ is derived from the discretization of the two-dimensional 
fan-beam Radon integral transform. Consequently, the inverse problem \eqref{eq:inverse_problem} necessitates the application of regularization techniques to guarantee a stable and reliable solution.
Since medical images often contain critical information in the form of low-contrast structures and boundaries between different anatomical tissues,  a widely used  regularization function is the popular isotropic Total Variation  ($TV$) operator \cite{rudin1992nonlinear}, defined as:
\begin{align}\label{eq:TV_definition}
		TV(\x) = \sum_{i=1}^n \sqrt{\left( \D_h \x \right)_i^2 + \left( \D_v \x \right)_i^2}.
\end{align}
where $\D_h, \D_v \in \R^{n \times n}$ are  discrete  difference operators associated with the horizontal and vertical derivatives, respectively.
TV is particularly effective in preserving edges, as it penalizes pixel intensity variations by promoting sparsity in the gradient domain, thereby maintaining sharp transitions \cite{sidky2014cttpv, piccolomini2021model, friot2022iterative,chan2020two}.
By employing TV  regularization, the reconstructed image is obtained as the solution of the following minimization problem over a suitable domain $\X \subseteq \R^n$:
\begin{align}\label{eq:minimization}
    \min_{\x \in \X} \mathcal{J}_\delta(\x)  :=  \|\K\x-\y^\delta\|_2^2+ \lambda TV(\x)
\end{align}
where the scalar $\lambda>0$ controls the contribution of the regularization term over the least-squares fidelity one, in the overall objective function $\mathcal{J}_\delta(\x)$.

In traditional regularized approaches, regularization is applied uniformly to the whole image, i.e. its formulation equally weights the contributions across the pixels. We refer to this case as {\em global regularization}. 
Global regularization is simple to implement and computationally low demanding; 
on the other hand, it treats the entire image as a homogeneous entity, making the setting of the optimal weight $\lambda$ not trivial conceptually because a value may not adapt well to varying characteristics within the image.
For instance, a considerably high value of $\lambda$ may effectively denoise the resulting image within flat regions, yet it could excessively smooth it, potentially leading to the loss or blurring of crucial details in certain areas. Conversely, a slightly lower value may effectively preserve fine details but might be insufficient in adequately suppressing noise in homogeneous regions.
A well-known drawback of considering the global TV regularization is the potential introduction of piecewise constant artifacts, which could impact the subsequent medical interpretation. It is particularly noticeable in regions with smooth intensity variations, as TV regularization tends to oversimplify these areas, resulting in blocky or stairstep artifacts. 

To improve the flexibility of regularization and achieve a more effective trade-off between denoising and preserving fine details, a promising approach is {\em space-variant regularization}, which involves weighting the regularization function pixel-wise. For this reason, this strategy is also referred to as {\em adaptive weighted regularization}. This approach, which adjusts the regularization strength according to local image properties, enhances the preservation of fine details and edges while sustaining effective noise reduction. However, a major challenge in the implementation of space-variant regularization is the selection of suitable local regularization parameters (also referred to as adaptive weights or space-variant parameters). 
The concept of local regularization parameters has been extensively explored in the literature. Here, we highlight a few key references, directing readers to them for further insights. 
A possible approach involves estimating the parameters based on the noise variance in different regions of the input image.
In a recent review by \cite{pragliola2023and}, the authors demonstrate the efficacy of space-variant TV-based regularization with weights derived from a Bayesian interpretation of the model.
In other studies, these weights are determined through various approaches, including estimating the noise variance \cite{dong2011automated, hintermuller2017optimalII, bortolotti2016uniform, cascarano2023constrained,kan2021pnkh}, analyzing image scaling properties \cite{grasmair2009locally, chen2010adaptive}.
More recently, several data-driven strategies have been proposed to estimate these parameters for a specific space-variant regularizer, or jointly with the regularizer itself, by training the models in different modalities \cite{bubba2022bilevel, cuomo2023speckle, kofler2023learning, pourya2024iteratively}.

All the previous references employ regularization techniques to address denoising and deblurring problems, wherein the observed image (datum) and the reconstructed image (solution) belong to the same space. However, our application differs in that the data (sinogram) resides in a lower-dimensional space of the solution; thus, the noise present in the sinogram differs significantly from the noise inherent in the reconstructed image. Indeed, streaking artifacts also affect the reconstructions, originating from the sparsity of the views.
 In this context, space-variant parameters are usually updated at each iteration of the iterative method employed to solve the problem \cite{sidky2014cttpv,huang2018scale,xi2023adaptive,luo2018adaptive}. This technique stems from the {\it iterative reweighted}  approach \cite{candes2008enhancing, daubechies2010iteratively}, commonly employed to tackle problems involving non-convex regularizers approximated by convex functions. However, to theoretically ensure global convergence, the iterative reweighting strategy necessitates computing the weights at each iteration via an additional nested minimization problem \cite{lazzaro2019nonconvex}. In practical applications, particularly in tomography, this step is often omitted to reduce computational time, thereby violating the assumptions needed for the convergence results.

\paragraph{Contributions.}
In this study, we propose an adaptive weighted total variation algorithm for tomographic image reconstruction from sparsely sampled views.
Unlike the reweighted approaches discussed in the aforementioned works, we do not update the coefficients throughout the iterative process, but we determine and fix them from the outset. This strategy enables the analysis of the convergence and of the regularization properties of the proposed model.  Indeed, we perform a theoretical analysis demonstrating that the proposed method qualifies as a well-posed regularization approach, according to \cite{scherzer2009variational}. 
\\
In addition, as it is not possible to leverage knowledge of noise in the data to estimate the adaptivity of the regularization because image and data lie in different domains in CT, the space-variant coefficients are not computed based on a noise estimation. Instead, they are determined using an intermediate image computed as a coarse and fast reconstruction.
The core of our proposal is a framework based on a neural network to compute this intermediate image. The network can be trained employing two alternative loss functions, to match either the image directly or the magnitude of its gradient.
Specifically, we demonstrate that if the intermediate image (or the magnitude of its gradient) accurately approximates the exact image (or the magnitude of its gradient, respectively), then the resulting adaptive weights lead to a highly accurate final reconstruction.\\
Finally, extensive numerical simulations confirm the effectiveness of the proposed method, even when compared to state-of-the-art approaches.

\paragraph {Structure of the paper.}

The structure of the paper is as follows: Section \ref{sec:model} introduces the proposed weighted Total Variation model, details the selection of adaptive weights, presents a preliminary experiment for validation and describes the proposed data driven approach to compute the weights. Section \ref{sec:wellposedness} provides a theoretical analysis of the proposed method. Section \ref{sec:numres} presents experimental results on both synthetic and real datasets and the conclusions are summarized in Section \ref{sec:concl}.
At last, Appendix \ref{app:proofs} derives preliminary results for the proofs of the theorems discussed in Section \ref{sec:wellposedness}.

\section{The proposed adaptive weighted TV model \label{sec:model}}

In this section, we define the adaptive weighted TV  model that we propose to investigate in this paper.
It is expressed as:
\begin{align}\label{eq:minimization_w}
    \min_{\x \in \X}  \|\K\x-\y^\delta\|_2^2+ \lambda TV_{\w}(\x)
\end{align}
where we consider  $\X$ as the  non-negative subspace of $\R^n$, and the weighted TV function $TV_{\w}(\x)$  is defined as:
	\begin{equation}\label{eq:TVweighted}
		TV_{\w}(\x) := \sum_{i=1}^n \w_i \sqrt{\left( \D_h \x \right)_i^2 + \left( \D_v \x \right)_i^2} = || \w \odot | \D \x | ||_1,
	\end{equation}
with $\w = (\w_1, \dots, \w_n) \in \R^n$  the vector of weights, and $\odot$  the element-wise product.
Recalling the notation already introduced in equation \eqref{eq:TV_definition}, we denote as
$| \D \x | \in \R^n$ is the gradient-magnitude image of $\x$, defined as:
	\begin{align}\label{eq:gradient_magnitude}
		\left( | \D \x | \right)_i =\sqrt{\left( \D_h \x \right)_i^2 + \left( \D_v \x \right)_i^2},
	\end{align}
whereas the 2-dimensional discrete gradient operator $\D: \R^n \to \R^{2n}$ is set in the following as: 
\begin{align}\label{eq:Dx_definition}
    \D\x = \begin{bmatrix}
        \D_h \x \\ \D_v \x
    \end{bmatrix} \in \R^{2n}.
\end{align}
With a slight abuse of notation, from \eqref{eq:Dx_definition}, it holds that $(\D\x)_i = (\D_h\x)_i$ for $i = 1, \dots, n$, while $(\D\x)_i = (\D_v\x)_{i-n}$ for $i=n+1, \dots, 2n$.



The selection of the weights $\w$ in the space-variant TV model \eqref{eq:TVweighted}  is a highly delicate and critical process.
Ideally, $\w_i$ should be small if the corresponding pixel $\x_i$ of the image to be restored lies either close to an edge of a low-contrast region or on a small detail, so that the regularization strength on that pixel is moderate and the fidelity term recovers the information provided in the data $\y^\delta$. 
Similarly, to maximize the regularization effect, $\w_i$ should be high on large and uniform regions of $\x$, to suppress the streaking artifacts and to avoid the noise propagation.\\
To appropriately select the weights, we build upon the iterative reweighting 
\(\ell_1\)-norm strategy \cite{candes2008enhancing,daubechies2010iteratively},  which is widely employed for solving Total \(p\)-norm Variation (TpV) problems, with $0<p<1$.
As specified in the Introduction, in iterative reweighting techniques, the parameters change at each iteration of the optimization solver. For instance, in \cite{sidky2014cttpv}, the weights are chosen as:
 \begin{equation}\label{eq:weights_ir}
   \w({\x^{(k)}}) :=  \left( \frac{\eta}{\sqrt{\eta^2 + | \D \x^{(k)}} |^2} \right)^{1 - p},
 \end{equation}
where the non-negative parameter $\eta$ prevents numerical instabilities, and $k$ refers to the current iteration number.
Inspired by this rule, we suppose to have an image $\tilde{\x}$ obtained as a coarse reconstruction from the data $\y^{\delta}$ and we construct the weight matrix before running the iterative process as:
	\begin{equation}\label{eq:OurWeights}
		\w(\tilde{\x}) :=  \left( \frac{\eta}{\sqrt{\eta^2 + | \D \tilde{\x} |^2}} \right)^{1 - p},
	\end{equation}
with a suitable parameter $\eta>0$ and a fixed value of $0<p<1$. \\
The following statements serve to elucidate the conceptual underpinnings of our weight parameters, formally.

\begin{proposition}\label{prop:w_is_a_scale_term}
    For any $\eta > 0$, any $p\in(0,1)$, and any $\tilde{\x} \in \R^n$, 
    $\left(w(\tilde{\x})\right)_i \in (0, 1]$,  $\forall i = 1, \dots, n$. Moreover, $\left(\w(\tilde{\x})\right)_i = 1$ on a pixel $i \in 1, \dots, n$ if and only if $(| \D \tilde{\x} |)_i = 0$.
\end{proposition}

\begin{proof}
    Note that $\left(\w(\tilde{\x})\right)_i>0$ since it is the ratio of strictly positive quantities. In addition,  it can be observed that $\left(\w(\tilde{\x})\right)_i$ is a decreasing function of $(| \D \tilde{\x} |)_i$ since $p\in(0,1)$ and its exponent is $1 - p > 0 $. This ensures that its maximum is $1$ and it is attained for $(| \D \tilde{\x} |)_i = 0$ .
\end{proof}
We observe that when $(| \D \tilde{\x} |)_i = 0$ (i.e., when $\tilde{\x}_i$ lies within a flat region), it follows that $ \left( \w(\tilde{\x}) \right)_i = 1 $, resulting in a local component of the TV being precisely weighted by $\lambda$. Conversely, when $(| \D \tilde{\x} |)_i \gg 0 $ (indicating that pixel $i$ is located at an edge or corresponds to a fine detail), $\left( \w(\tilde{\x}) \right)_i < 1$. This implies that the local contribute in the TV regularization is reduced, thereby allowing the preservation of fine details. Figure \ref{fig:eta_plot} depicts  the value of $\left(\w(\tilde{\x})\right)_i$  as a function of $(| \D \tilde{\x} |)_i$, for different values of $\eta > 0$. 
It is evident that for very small values of $\eta$, $\left(\w(\tilde{\x})\right)_i$ rapidly decreases to zero (i.e., the solution of the problem is not regularized around pixel $i$). On the other side, for large values of $\eta$, $\left(\w(\tilde{\x})\right)_i$ decreases very slowly as a function of $ (| \D \tilde{\x} |)_i$, which implies that small details will be flattened out.

\begin{figure}[ht]
    \centering
    \includegraphics[width=0.6\linewidth]{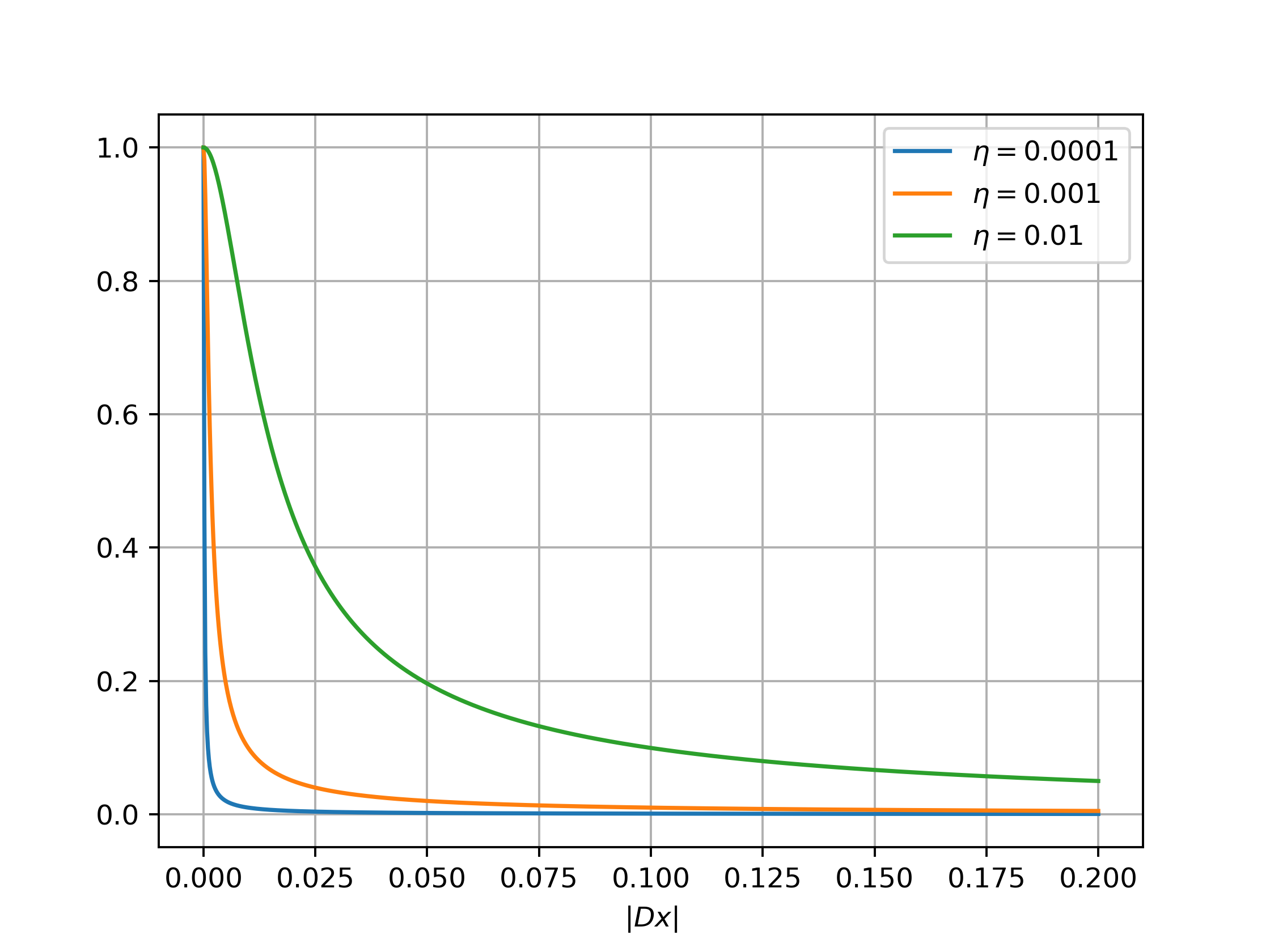}
    \caption{A plot of $\left(\w(\tilde{\x})\right)_i$ for different values of $\eta$, over $ (| \D \tilde{\x} |)_i$, \EM{for $p=0.3$}.}
    \label{fig:eta_plot}
\end{figure}

To provide a more precise specification on how to derive $\Tilde \x$ from $\y^{\delta}$  we consider $\Psi: \mathbb{R}^m \rightarrow \mathbb{R}^n$ as a Lipschitz-continuous function that maps $\y^{\delta}$ to an approximate reconstruction $\tilde{\x} = \Psi(\y^{\delta})$ of $\x^{GT}$. In accordance with the nomenclature introduced in \cite{evangelista2022or}, we denote such mappings as {\it reconstructors}, and we name as $Rec_{m, n}$ the set of such reconstructors.\\
In the following, we introduce a specific variant of \eqref{eq:minimization_w} where the weights are chosen according to \eqref{eq:OurWeights} relying on the reconstruction $\tilde{\x}=\Psi(\y^{\delta})$. For this reason,  we refer to this model as $\Psi$-W$\ell_1$ to remark its reliance on $\Psi$ in the weighted $\ell_1$-norm regularization. 
Formally, it reads:

\begin{equation}\label{eq:Psi_Wl1_formulation}
    \x^*_{\Psi, \delta} \in \arg\min_{\x \in \X} \mathcal{J}_{\Psi, \delta}(\x) := || \K \x - \y^\delta ||_2^2 + \lambda || \w(\Psi(\y^\delta)) \odot | \D \x | \> ||_1.
\end{equation}
In the following we will denote the regularization term of \eqref{eq:Psi_Wl1_formulation} as
$\mathcal{R}_{\Psi, \delta}(\x)$.\\
To address the optimization problem \eqref{eq:minimization_w}, we employ the Chambolle-Pock (CP) algorithm \cite{chambolle2011first}, a widely used iterative method. This approach is particularly popular for solving $\text{TV}$-regularized inverse problems, including applications in few-view CT reconstruction \cite{sidky2014cttpv, piccolomini2021model}. 
In its original formulation, the Chambolle-Pock method \cite{chambolle2011first}  was introduced to minimize an objective function of the form:
\begin{equation}\label{eq:CP_general_formulation}
    \min_{\x \in \R^n} F(\M\x) + G(\x),
\end{equation}
where both $F$ and $G$ are real-valued, proper, convex, lower semi-continuous functions and $\M$ is a linear operator from $\R^n$ to $\R^s$. 
Note that there are no constraints on the smoothness of either $F$ and $G$; therefore, the method can be applied to our problem by setting:
\begin{align}\label{eq:ourproblem}
    \begin{cases}
        G(\x) = \iota_{\X}(\x), \\
        F(\M\x) = \mathcal{J}_{\Psi, \delta}(\x) = \frac{1}{2} || \K\x - \y^\delta ||_2^2 + \lambda || \w(\Psi(\y^\delta)) \odot | \D \x | ||_1,
    \end{cases}
\end{align}
where $\iota_{\X}(\x)$ is the indicator function of the feasible set $\X$. As already stated, we consider $\X$ as the non-negative subspace of $\R^n$.\\
To apply the CP method to our problem, we define the linear operator $\M \in \R^{s \times n}$  by concatenating row-wise $\K$ and $\D$, namely $\M = \left[ \K; \D \right]$ and $s=(m + 2n)$. The CP algorithm considers the primal-dual formulation of \eqref{eq:CP_general_formulation}, which reads:
\begin{align}
    \min_{\x \in \R^n} \max_{\z \in \R^{m+2n} } \z^T \M \x + G(\x) - F^*(\z),
\end{align}
where $F^*$ is the convex conjugate of $F$ \cite{bauschke2017correction}, defined as:
\begin{align}
    F^*(\z^*) := \sup_{\z \in \R^{m+2n} } \left\{ \z^T \z^* - F(\z) \right\}.
\end{align} 
More details on the CP implementation can be found in \cite{sidky2014cttpv}. 

We consider here different possible choices to define the reconstructor $\Psi(\y^\delta)$ and subsequently compute  $\tilde{\x}$ from the undersampled sinogram $\y^\delta$.
Dealing with tomographic imaging from  X-ray measurements, a simple yet fast reconstructor is the Filtered Back Projection (FBP) solver \cite{kak2001principles}. As an analytical algorithm, it operates by applying a filter to the back-projected data, disregarding the inverse problem formulation of the imaging task. Even if FBP may not address all reconstruction challenges, its speed and simplicity make it a pragmatic choice for many medical imaging scenarios, particularly when balancing computational resources and reconstruction quality is essential. We denote the resulting method as  FBP-$W\ell_1$.
A further option is setting $\Psi(\y^{\delta})$ as an early-stopped iterative scheme which minimizes a variational problem.  We still consider the global TV regularization \eqref{eq:TV_definition} and stop the iterative  CP algorithm for the solution  of the corresponding problem \eqref{eq:minimization} after very few iterations. We denote the resulting method as  TV-$W\ell_1$.
Figure \ref{sfig:GrAbstract1} illustrates the proposed framework, highlighting that the weights are computed based on an intermediate image, $\tilde{\x}$.

\begin{figure}[ht]
\centering
\includegraphics[trim=0 45mm 0 0,clip, width=0.8\textwidth]{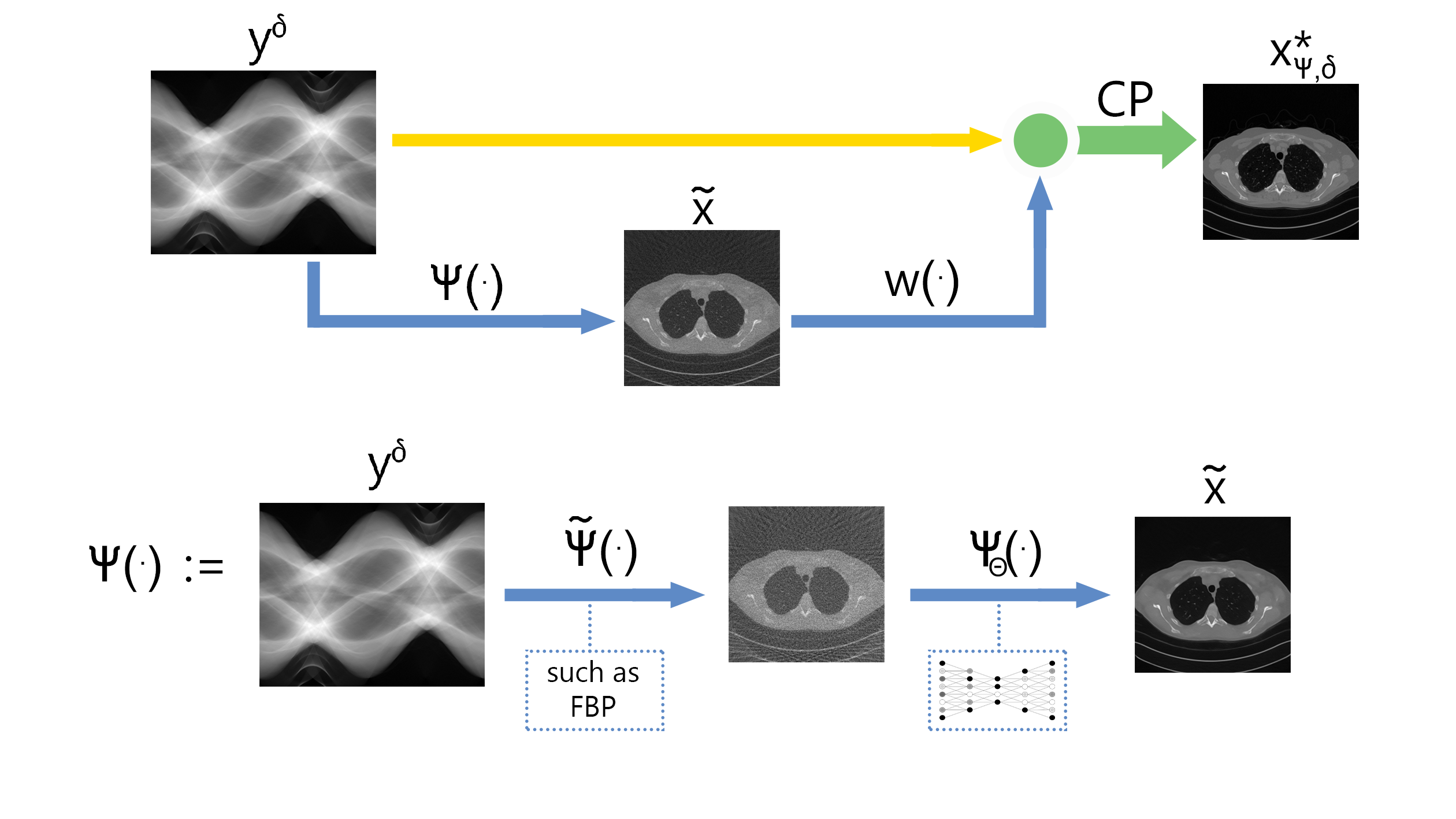}
\caption{ Workflow of the considered scheme, where the reconstructor $\Psi$ is used to achieve a useful image $\tilde{\x} = \Psi(\y^\delta)$ for adaptive $\ell_1$ regularization.} \label{sfig:GrAbstract1}
\end{figure}

\subsection{Experiments on a synthetic image}\label{ssec:Sintetica}

To analyze the behavior of the proposed weighting scheme for different choices of $\tilde{\x}$, we conduct a simple experiment involving sparse CT reconstruction using a synthetic image, shown as the first image in Figure \ref{fig:results_Sint02}. We denote as $\x^{GT}$ this synthetic ground truth  which incorporates structures commonly encountered in tomographic imaging: homogeneous regions with regular shapes simulating tumoral masses, high-density areas representing bones or metallic implants, and objects with extremely thin edges.\\
The corrupted measurement $\y^{\delta}$ is then computed by simulating projections from the ground truth image $\x^{GT}$, according to the image formation model in Equation \eqref{eq:inverse_problem} using a fan-beam scanning geometry with only 45 projections uniformly distributed over the angular range $[0, 180]$. The noise $\e$ is obtained by sampling $\z \sim \mathcal{N}(\boldsymbol{0}, \I)$ and computing:
\begin{align*}
    \e = \nu \frac{|| \y ||_2}{|| \z ||_2}\z,
\end{align*}
where we indicated with $\y$ the noiseless sinogram. Note that this formulation is equivalent to \eqref{eq:inverse_problem}, with $\delta = \nu || \y ||_2$.\\
We analyze the reconstructions by means of metrics such as the Relative Error (RE), the Peak Signal to Noise Ratio (PSNR) and the Structural Similarity Index (SSIM) defined in \cite{wang2004image},  with respect to the ground truth image.\\ 
We conducted two simulations with different noise intensity, setting $\nu=0.005$ and $\nu=0.02$. For all the experiments, we have chosen \(\eta = 2 \cdot 10^{-5}\).
Regarding the regularization parameter $\lambda$, we have assigned consistent heuristic values across all weighted methods.
To validate the selections of the space-variant parameters, we compare the results obtained using the proposed adaptive weighted TV model \eqref{eq:TVweighted} method with those derived from global TV regularization \eqref{eq:minimization} (the value of the regularization parameter $\lambda$  has been adequately set {\em ad hoc} for the global TV). 
We briefly present the results of this simulation, which firstly demonstrate the effectiveness of the proposed weights and subsequently confirm that setting \(\tilde{\x} = \x^{GT}\) yields a highly accurate reconstructed image. \\
Table \ref{tab:Sintetica} presents the metrics for the intermediate image $\tilde{\x}$ and the final reconstructed image $\x^*_{\Psi,\delta}$. We denote as GT-$W\ell_1$ the adaptive weighted TV  method where $\tilde{\x} = \x^{GT}$. When compared to the global TV, the results underscore that the proposed adaptive regularization method consistently produces high-quality reconstructions, as indicated by SSIM values exceeding 0.99. Furthermore, the analysis of the relative errors and PSNR values demonstrates superior performance when the intermediate reconstruction is set as $\tilde{\x} = \x^{GT}$.

\begin{table}[h]
\caption{Performance results on the synthetic image with different noise levels on the sinogram ($\nu=0.005$ and $\nu=0.02$). In the first three columns, the metrics are relative to the image $\tilde{\x}$, and in the last three columns the metrics are relative to the output image $\x^*_{\Psi,\delta}$.}\label{tab:Sintetica}%
\resizebox{\textwidth}{!}{%
\begin{tabular*}{\textwidth}{@{\extracolsep\fill}ll rrr rrr}
\toprule
 &   & \multicolumn{3}{@{}c@{}}{ $\tilde{\x} $} & \multicolumn{3}{@{}c@{}}{$\x^*_{\Psi,\delta}$}\\
  &         & RE & PSNR   & SSIM          & RE & PSNR   & SSIM\\
\cmidrule{3-5} \cmidrule{6-8}
\midrule
\multirow{ 4}{*}{$\nu=0.005$} 
   & GT-$W\ell_1$  & 0.0000   & 100.00  & 1.0000            & 0.0217 & 46.9881 & 0.9979  \\
   & FBP-$W\ell_1$   & 0.5056   & 19.6281  & 0.1974           & 0.0373 & 42.2662 & 0.9953  \\
   & TV-$W\ell_1$    & 0.2236   & 26.7167  & 0.7335           & 0.0621 & 37.8454 & 0.9920  \\
   & global TV    & -   & -  & -                    & 0.1236 & 31.8657 & 0.9805  \\
\midrule
\multirow{ 4}{*}{$\nu=0.02$}
   & GT-$W\ell_1$   & 0.0000   & 100.00  & 1.0000            & 0.0398 & 41.7159 & 0.9968  \\
   & FBP-$W\ell_1$    & 0.9694   & 13.9748 & 0.0524            & 0.1025 & 33.4928 & 0.9809  \\
   & TV-$W\ell_1$   & 0.2435 & 25.9754 & 0.5968            & 0.0911 & 34.5162 & 0.9797  \\
   & global TV    & -   & -  & -                     & 0.1249 & 31.7715 & 0.9787  \\
\bottomrule
\end{tabular*}
}
\end{table}

These findings are further confirmed by Figure \ref{fig:results_Sint02}, which presents a cropped view of the restored images in the case $\nu=0.02$. Clear differences in the outputs are noticeable, particularly in the reconstruction of finer cross details and contrast. Notably, the global TV method consistently fails to reconstruct the cross, regardless of the parameter settings. The superior reconstruction performance achieved when $\tilde{\x} = {\x}^{GT}$ highlights the effectiveness of adapting regularization to image pixels when a highly accurate approximation of the target image is available.
Figure \ref{fig:results_Sint02} also includes the plot of the relative error, where the relative position of the curves is maintained across all iterations, consistently with the behavior discussed in Table \ref{tab:Mayo005}. The plots clearly demonstrate that the error stabilizes, suggesting that a stationary point has been achieved for each method.
We specify that the graphs display the executions carried out with 10000 iterations of the Chambolle-Pock method to analyze the behavior until convergence, whereas the images visually show no further changes after a few hundred iterations.

\begin{figure}[ht]
\centering
\subfloat[$\x^{GT}$]{
\begin{tikzpicture}
        \node [anchor=south west, inner sep=0] (image) at (0,0) {\includegraphics[width=0.18\linewidth]{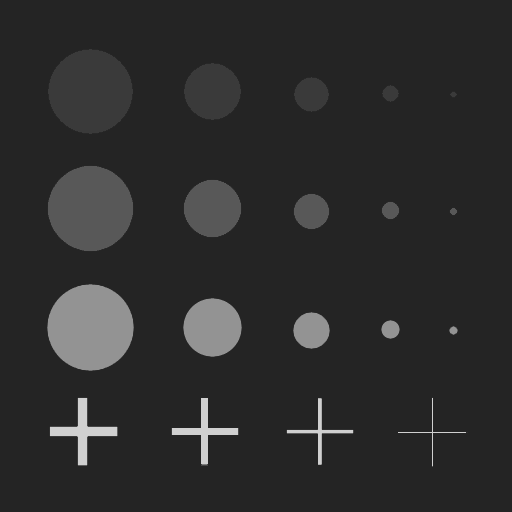}};
        \begin{scope}[x={(image.south east)}, y={(image.north west)}]
            \draw[red, thick] (0.52, 0.03) rectangle (0.97, 0.48);
        \end{scope}
    \end{tikzpicture}
}
\subfloat[GT-$W\ell_1$]{
\includegraphics[trim={65mm 0 0 65mm},clip, width=0.18\linewidth]{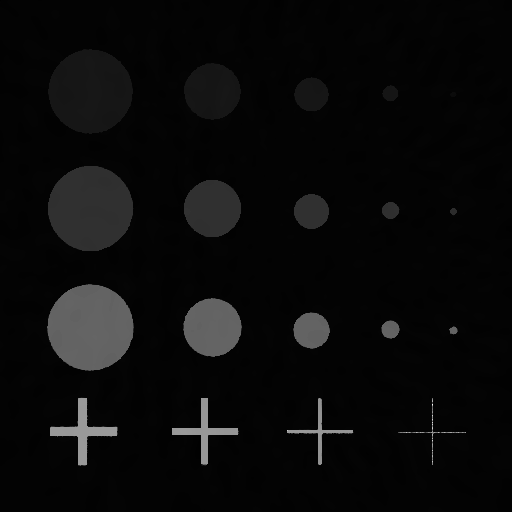}
}
\subfloat[FBP-$W\ell_1$]{
\includegraphics[trim={65mm 0 0 65mm},clip, width=0.18\linewidth]{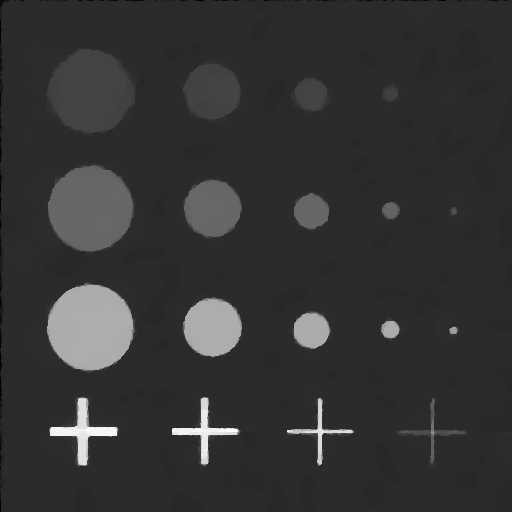}
}
\subfloat[TV-$W\ell_1$]{
\includegraphics[trim={65mm 0 0 65mm},clip, width=0.18\linewidth]{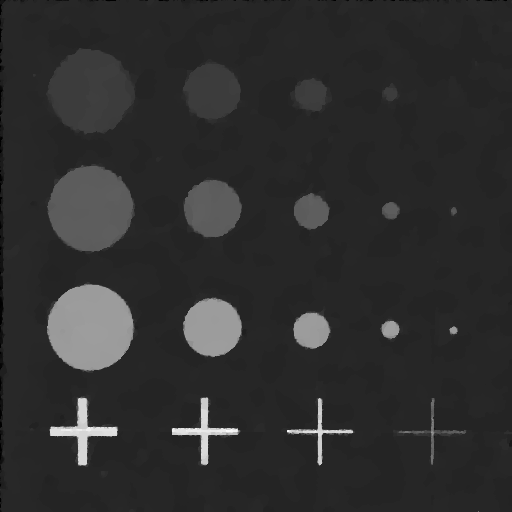}
} 
\subfloat[global TV]{
\includegraphics[trim={65mm 0 0 65mm},clip, width=0.18\linewidth]{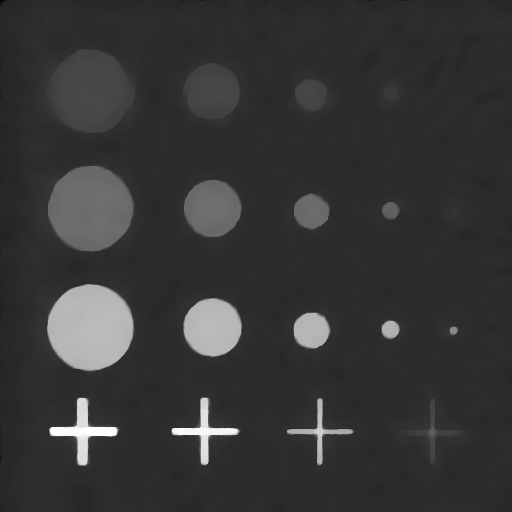} 
}\\
\subfloat[Relative error]{
\includegraphics[width=0.4\linewidth]{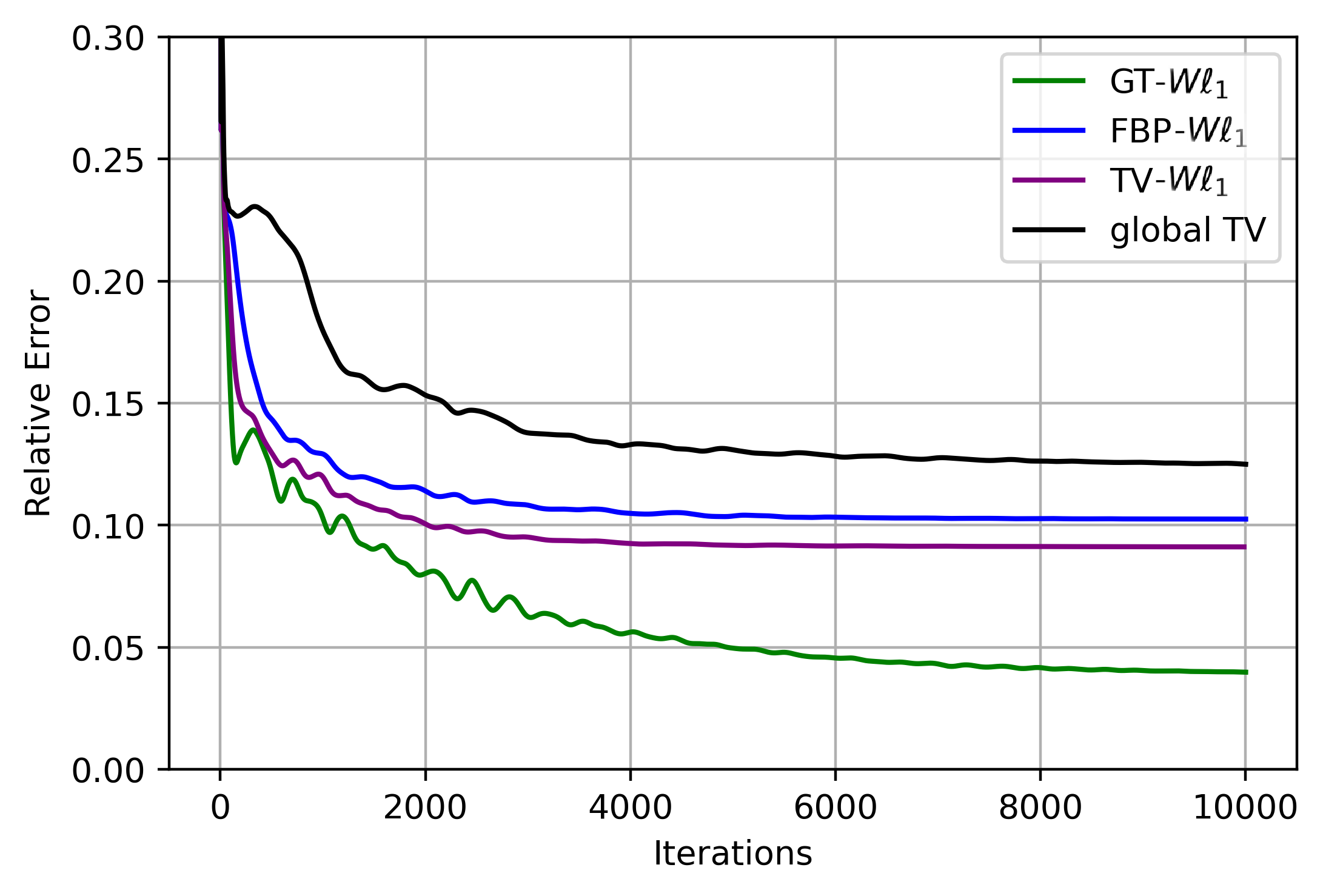}
}
  \caption{Results of the experiment on the synthetic image with higher noise ($\nu=0.02$). 
  In the first row, from left to right: the entire ground truth image with a red square depicting the crop of interest, cropped zooms on the reconstructions by GT-$W\ell_1$, FBP-$W\ell_1$, TV-$W\ell_1$ and by the global TV model.
  In the second row: plot of the Relative Error  over the iterations.}
  \label{fig:results_Sint02}
\end{figure}

\subsection{Weights Computation via a Neural Network \label{ssec:nn}}

We boost now our novel approach for weights computation by exploiting as the intermediate image  $\tilde \x$ the output of a neural network, trained in a supervised learning framework. 
 As the network, we use in our experiments the renowned Residual U-Net architecture as described in \cite{green_post_processing,evangelista2023ambiguity}.
In order to train the neural network, we consider the dataset $\mathcal{D} = \{ (\tilde \Psi(\y_i^{\delta}), \x^{GT}_i) \}_{i=1}^{N_D}$, where $\tilde \Psi(\y_i^{\delta})$ is a fast coarse reconstructor. In our experiments, we will use as $\tilde \Psi$ the FBP algorithm.
Training a neural network results in finding the parameters $\theta^*$ as a result of the minimization of a loss function. 
Motivated by the preliminary experiment in Section \ref{ssec:Sintetica} we consider two distinct loss functions to train our network. The first one, named in the following as image loss, approximates the ground truth images $\x^{GT}_j$, and the parameters $\theta^*$ are computed as:

\begin{equation}\label{eq:mseloss}
        \theta^* \in \arg\min_{\theta} \sum_{j=1}^{N_D} || \x^{GT}_j - \Psi_\theta(\tilde \Psi(\y^\delta_j)) ||_2^2.
    \end{equation}
    
The choice of the second loss function is motivated by the idea of "learning" the image gradients, as these are explicitly used in the computation of the weights, and for this reason it is named gradient loss in the following. In the next section, this choice will also be justified through theoretical results.
In this case the parameters $\theta^*$ are computed as:

    \begin{equation}\label{eq:gradloss}
        \theta^* \in \arg\min_{\theta} \sum_{j=1}^{N_D} || |\D \x^{GT}_j| - |\D \Psi_\theta(\tilde \Psi (\y^\delta_j))| ||_2^2.
    \end{equation}

The resulting scheme considers the reconstructor $\Psi\in Rec_{m, n}$ as the neural network applied to the FBP reconstruction of the data $\y_i^{\delta}, i=1, \ldots N_D$.
We illustrate the proposed framework, that we name as $\Psi$-$W\ell_1$, in Figure \ref{fig:GrAbstract2}.


\begin{figure}
\centering
\includegraphics[trim=0 0 0 45mm,clip, width=0.8\textwidth]{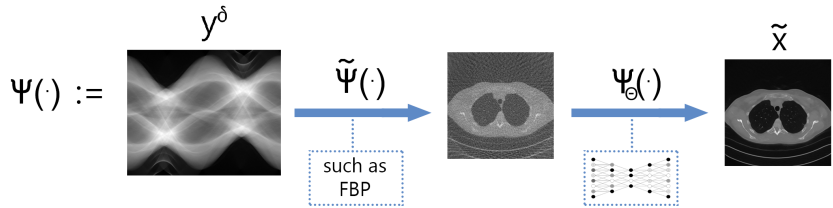}
\caption{The reconstructor $\Psi$ when it is constituted by two steps: a Filtered Back Projection and a neural network.} \label{fig:GrAbstract2}
\end{figure}

\section{On the well-posedness of the proposed regularization method}\label{sec:wellposedness}
In this section, we establish the theoretical foundation of our proposed $\Psi$-$W\ell_1$ approach, as formulated in \eqref{eq:Psi_Wl1_formulation}, by demonstrating that it constitutes a well-posed regularization method. To this end, following the analysis presented in \cite{scherzer2009variational}, we will demonstrate that the $\Psi$-$W\ell_1$ approach satisfies the following properties:

\begin{itemize}
    \item \emph{Existence:} for any  $\lambda > 0$, any reconstructor $\Psi 
    \in Rec_{m, n}$, and any $\delta \geq 0$, the objective function $\mathcal{J}_{\Psi, \delta}(\x)$ defined in \eqref{eq:Psi_Wl1_formulation} admits at least a minimizer $\x^*_{\Psi, \delta}$.
    \item \emph{Uniqueness:} for any   $\lambda > 0$, any $\Psi \in Rec_{m, n}$, and any $\delta \geq 0$, the minimizer $\x^*_{\Psi, \delta}$ of $\mathcal{J}_{\Psi, \delta}(\x)$ is unique,
    \item \emph{Noise Stability:}  given   $\lambda > 0$ and  a reconstructor $\Psi \in Rec_{m, n}$, if $\delta \to 0$, then $\x^*_{\Psi, \delta} \to \x^*_{\Psi, 0}$.
    \item \emph{Reconstructor Stability:}  given  $\lambda > 0$ and $\delta \geq 0$, if $\Psi \approx  \Psi^*$ , then $\x^*_{\Psi, \delta} \to \x^*_{\Psi^*, \delta}$.
\end{itemize}

Our theoretical analysis offers a novel perspective on regularization theory for inverse problems. Unlike traditional approaches, where the regularization term depends solely on the  variable $\x$, we propose a framework where it  also depends on the noisy measurements $\y^\delta$.
Additionally, while the uniqueness of the minimizer under suitable conditions is well-established for anisotropic Total Variation, we extend this result to isotropic Total Variation-regularized inverse problems. This extension arises as a special case of our general uniqueness theorem, where the weights are uniformly set to 1.
Some of the topics discussed are introduced in \cite{scherzer2009variational} in very general settings. They have been reformulated with natural assumptions tailored to the problem considered in this paper.

The main theorems that establish the previously stated four properties are presented and proved in the following.
The supporting lemmas used in the proofs are stated here but their proofs are in Appendix \ref{app:proofs}.\\
We remark that the majority of the theoretical results rely on the following assumption:

 \begin{enumerate}[label=(A.\arabic*), ref=(A.\arabic*)]
    \item \label{assumption:kernels} $\ker(\K) \cap \ker(\D) = \{ \boldsymbol{0} \}$. 
\end{enumerate}
Note that when $\K$ represents the CT projection operator, Assumption \ref{assumption:kernels} is not restrictive. This is because $\ker(\D)$ corresponds to the set of constant images, which cannot lie in the kernel of $\K$, as: 
\begin{align}
    \y_i = (\K \x)_i = \sum_{j \in L_i} \x_j,
\end{align}

where $L_i$ parametrizes the trajectory of the $i$-th X-ray  \cite{bertero2021introduction}. Therefore, $\y_i = 0$ if and only if $\sum_{j \in L_i} \x_j = 0$ which, given that $\x_j \geq 0$ by hypothesis, holds if and only if $\x_j = 0$ for any $j \in L_i$. Moreover, if $\x \in \ker (\D)$ therefore $\x_j = 0$ for any $j = 1, \dots, n$. Consequently, we have demonstrated that $\ker(\K) \cap \ker(\D) = \{ \boldsymbol{0} \}$, satisfying \ref{assumption:kernels}. From now on, we will always assume \ref{assumption:kernels} is satisfied.

\begin{lemma}\label{lemma:kerWD_equal_kerD}
        Let $\boldsymbol{W}_{\Psi, \delta} \in \R^{2n \times 2n}$ be the diagonal matrix having two copies of $\w(\Psi(\y^\delta))$ on the diagonal, so that: 
        \begin{align}
            \w(\Psi(\y^\delta)) \odot | \D \x | = | \boldsymbol{W}_{\Psi, \delta} \D \x |.
        \end{align}
        Then, it is trivial to prove that:
        \begin{align}
            \ker(\boldsymbol{W}_{\Psi, \delta} \D) = \ker(\D).
        \end{align}
    \end{lemma}


  This result, together with Assumption \ref{assumption:kernels}, implies that $\ker (\K) \cap \ker(\boldsymbol{W}_{\Psi, \delta} \D) = \{ \boldsymbol{0} \}$. Therefore, we can prove the next Lemma.

    \begin{lemma}\label{lemma:coercivity}
        For any $\delta \geq 0$ and any $\Psi$, the objective function $\mathcal{J}_{\Psi, \delta}$  in \eqref{eq:Psi_Wl1_formulation} is coercive.
    \end{lemma}

    \begin{proof}
        The proof is in Appendix \ref{app:proofs}
    \end{proof}

    We can now state and prove the theorem on the existence of a minimizer of \eqref{eq:Psi_Wl1_formulation}.
    
\begin{theorem}[Existence] \label{prop:existence}
    For any $\delta \geq 0$ and any $\Psi \in Rec_{m, n}$, $\mathcal{J}_{\Psi, \delta}(\x)$ admits at least a minimizer $\x^*_{\Psi, \delta}$.
\end{theorem}

\begin{proof}
    Note that $\mathcal{J}_{\Psi, \delta}(\x)$ is convex as a function of $\x$, as it is the sum of two convex functions $|| \K \x - \y^\delta||_2^2$ and $\mathcal{R}_{\Psi, \delta}(\x)$. Moreover, by Lemma \ref{lemma:coercivity}, $\mathcal{J}_{\Psi, \delta}(\x)$ is coercive. Therefore, it has at least a minimizer $\x^*_{\Psi, \delta}$.
\end{proof}

We define the following set:
\begin{align}
    \mathcal{M} := \{ \x^*_{\Psi, \delta} \in \X \> | \> \x^*_{\Psi, \delta} \mbox{ is a minimizer of } \mathcal{J}_{\Psi, \delta}(\x) \}
\end{align}
and, for any $\x \in \mathcal{M}$, we define  $\bar{I}(\x) = \{ i \in \{1, \dots, n\} \ | \ \left(| \D \x |\right)_i = 0 \}$  as the set of indices where $| \D \x |$ is zero, and   $S(\x):= \{ \boldsymbol{v} \in \R^n \: | \:  \ \left(| \D \boldsymbol{v}|\right)_i = 0\ \forall\ i\in\bar{I}(\x) \}$ as the set of vectors whose gradient is zero wherever the gradient of $\x$ is zero. 

To prove the uniqueness, we 
 need to derive the optimality conditions for $\mathcal{J}_{\Psi, \delta}(\x)$. Let $\partial \mathcal{J}_{\Psi, \delta}(\x)$ be the subdifferential of $\mathcal{J}_{\Psi, \delta}(\x)$, computed on $\x$. By the Fermat Theorem, if $\x^*_{\Psi, \delta}$ is a minimizer of $\mathcal{J}_{\Psi, \delta}(\x)$, then necessarily $\boldsymbol{0} \in \partial \mathcal{J}_{\Psi, \delta}(\x^*_{\Psi, \delta})$. By the linearity of the subgradient, it is not hard to show that:

\begin{align}
\begin{split}
    \partial \mathcal{J}_{\Psi, \delta}(\x) &= \left\{
        2 \K^T(\K \x - \y^\delta) + \lambda \boldsymbol{v}\> | \>\boldsymbol{v}\in\partial \mathcal{R}_{\Psi, \delta}(\x)
    \right\}.
\end{split}
\end{align}

To compute  $\partial \mathcal{R}_{\Psi, \delta}(\x)$ we first prove the following lemma:

\begin{lemma}\label{lemma:TV_gradient}
        Let $TV(\x) = \sum_{i=1}^n \sqrt{\left( \D_h \x \right)_i^2 + \left( \D_v \x \right)_i^2}$, where $\D_h, \D_v$ are defined as in \eqref{eq:TV_definition}.
    Moreover, define the vector $\boldsymbol{g} \in \R^{2n}$ such that:
    \begin{align}\label{eq:g_definition}
        \boldsymbol{g}_i = \begin{cases}
            \frac{(\D_h \x)_i}{(| \D \x |_i)} & \mbox{if } i \in \{1, \dots, n\} \cap \left( \bar{I}(\x) \right)^c, \\
            \frac{(\D_v \x)_i}{(| \D \x |_i)} & \mbox{if } i \in \{n+1, \dots, 2n\} \cap \left( \bar{I}(\x) \right)^c, \\
            c & \mbox{if } i \in \bar{I}(\x). \\
        \end{cases}
    \end{align}
    Then:
    \begin{align}
        \partial TV(\x) = \{ \D^T \boldsymbol{g} \text{ with } \boldsymbol{g}\in\R^{2n} \mbox{ defined as in  \eqref{eq:g_definition} with } c \in [-1, 1]\}. 
    \end{align}
\end{lemma}

\begin{proof}
    The proof is in  Appendix \ref{app:TV_gradient}.
\end{proof}

Exploiting the characterization of $\partial TV(\x)$ reported in Lemma \ref{lemma:TV_gradient} and recalling the definition of our weighted regularization term:
\begin{equation*}
     \mathcal{R}_{\Psi, \delta}(\x) = || \w(\Psi(\y^\delta)) \odot | \D \x | \> ||_1 = \sum_{i=1}^{n}(\w(\Psi(\y^\delta)))_{i} \sqrt{(\D_h \x)^2_{i} + (\D_v \x)^2_{i}}, 
\end{equation*}
we can easily prove that $\partial\mathcal{R}_{\Psi, \delta} (\x) = \left\{\w(\Psi(\y^\delta)) \odot \boldsymbol{v}\> | \>\boldsymbol{v}\in \partial TV(\x)\right\}$ or equivalently
\begin{equation}
    \partial\mathcal{R}_{\Psi, \delta}(\x) = \{ \w(\Psi(\y^\delta)) \odot \D^T \boldsymbol{g}\> |\> \boldsymbol{g} \mbox{ defined as in  \eqref{eq:g_definition}}\}.
\end{equation}

Therefore, we can derive the optimality conditions for the variational problem \eqref{eq:Psi_Wl1_formulation}.

\begin{lemma}\label{lemma:optimality_condition}
    If $\x^*_{\Psi, \delta}$ is a minimizer of $\mathcal{J}_{\Psi, \delta}$, then there exists $\bar{\boldsymbol{g}}\in\R^{2n}$ such that $\D^T \bar{\boldsymbol{g}} \in \partial TV(\x^*_{\Psi, \delta})$ and satisfies the following condition:
    \begin{align}
        \w(\Psi(\y^\delta)) \odot \D^T\bar{\boldsymbol{g}} = - \frac{2}{\lambda} \K^T (\K \x^*_{\Psi, \delta} - \y^\delta).
    \end{align}
\end{lemma}

\begin{proof}
    The proof trivially follows observing that $\boldsymbol{0}\in\partial \mathcal{J}_{\Psi, \delta} (\x^*_{\Psi, \delta})$.
\end{proof}

The optimality conditions represent one of the two ingredients necessary to prove the uniqueness of $\x^*_{\Psi, \delta}$. The next step is to derive some properties of the objective function $\mathcal{J}_{\Psi, \delta}(\x)$ on elements of $\mathcal{M}$. Their proof is again obtained through a sequence of lemmas, some of which are similar to results presented e.g. in \cite{jorgensen2015testable}.

\begin{lemma}\label{lemma:J_properties_on_M}
    For any $\x_1, \x_2 \in \mathcal{M}$, the following holds:
    \begin{enumerate}
        \item $\x_1-\x_2 \in \ker(\K)$,
        \item $\mathcal{R}(\x_1) =\mathcal{R}(\x_2)$.
    \end{enumerate}
\end{lemma}

\begin{proof}
    The proof is in Appendix \ref{app:J_properties_on_M}.
\end{proof}

We finally need this property of the set $S(\x)$ to prove the theorem.

\begin{lemma}\label{lemma:S1_vectorial_space}
    For any $\x \in \X$, the set $S(\x)$ is a vectorial space and $\x \in S(\x)$.
\end{lemma}

\begin{proof}
The proof is trivial.
\end{proof}

The following theorem states the condition to have a unique solution of the considered minimization problem.

\begin{theorem}[Uniqueness] \label{prop:unicity}
    For any $\delta \geq 0$ and any $\Psi \in Rec_{m, n}$, let $\x^*_{\Psi, \delta} \in \mathcal{M}$ be a minimizer of $\mathcal{J}_{\Psi, \delta}(\x)$. If:
    \begin{enumerate}
        \item $\ker(\K) \cap S(\x^*_{\Psi, \delta}) = \{ \boldsymbol{0} \}$,
        \item $\bar{\boldsymbol{g}}$, defined as in Lemma \ref{lemma:optimality_condition}, has $c \in (-1, 1)$,
    \end{enumerate} 
    then $\mathcal{M} = \{ \x^*_{\Psi, \delta} \}$, i.e. $\x^*_{\Psi, \delta}$ is the only minimizer of $\mathcal{J}_{\Psi, \delta}$.
\end{theorem}

\begin{proof}
    By contradiction, let $\x^*_{\Psi, \delta}, \x' \in \mathcal{M}$ such that $\x^*_{\Psi, \delta} \neq \x'$, and consider the set $S(\x^*_{\Psi, \delta})$ defined above.
    
    If $\x' \in S(\x^*_{\Psi, \delta})$, then $\x^*_{\Psi, \delta} - \x' \in S(\x^*_{\Psi, \delta})$ as well, since $S(\x^*_{\Psi, \delta})$ is a vectorial space as shown in Lemma \ref{lemma:S1_vectorial_space}. Moreover, by Lemma \ref{lemma:J_properties_on_M}, $\x^*_{\Psi, \delta} - \x' \in \ker(\K)$. Consequently, since $\ker(\K) \cap S(\x^*_{\Psi, \delta}) = \{ \boldsymbol{0} \}$ by hypothesis, then $\x^*_{\Psi, \delta} = \x'$, contraddicting the assumption.

    If instead $\x' \notin S(\x^*_{\Psi, \delta})$, consider the vector $\bar{\boldsymbol{g}}\in \R^{2n}$, defined as in Lemma \ref{lemma:optimality_condition}. Note that, by hypothesis, it has $c \in [-1, 1]$, and $\D^T \bar{\boldsymbol{g}} \in \partial TV(\x^*_{\Psi, \delta})$ as shown in Lemma \ref{lemma:TV_gradient}. Defined $\boldsymbol{W}_{\Psi, \delta}$ as the $2n \times 2n$ diagonal matrix with two copies of $\w(\Psi(\y^\delta))$ as diagonal as in Lemma \ref{lemma:kerWD_equal_kerD}, it is not hard to show that $\D^T \boldsymbol{W}_{\Psi, \delta}^T \bar{\boldsymbol{g}} = \w(\Psi(\y^\delta)) \odot \D^T \bar{\boldsymbol{g}} \in \partial \mathcal{R}_{\Psi, \delta}(\x^*_{\Psi, \delta})$, the subdifferential of $\mathcal{R}_{\Psi, \delta}(\x)$ in $\x^*_{\Psi, \delta}$. Therefore, let $\mathcal{B}_{\mathcal{R}_{\Psi, \delta}}(\x', \x^*_{\Psi, \delta})$ be the Bregman distance between the two vectors $\x^*_{\Psi, \delta}, \x'$ associated with $\mathcal{R}_{\Psi, \delta}(\x)$ at $\D^T \boldsymbol{W}_{\Psi, \delta}^T \bar{\boldsymbol{g}}$. By Lemma \ref{lemma:J_properties_on_M},
    \begin{align}
    \begin{split}
        \mathcal{B}_{\mathcal{R}_{\Psi, \delta}}(\x', \x^*_{\Psi, \delta}) &= \underbrace{\mathcal{R}_{\Psi, \delta}(\x') - \mathcal{R}_{\Psi, \delta}(\x^*_{\Psi, \delta})}_{= 0} - \langle \D^T \boldsymbol{W}_{\Psi, \delta}^T \bar{\boldsymbol{g}}, \x' - \x^*_{\Psi, \delta} \rangle \\ &=  \langle \D^T \boldsymbol{W}_{\Psi, \delta}^T \bar{\boldsymbol{g}}, \x^*_{\Psi, \delta} - \x' \rangle.
    \end{split}
    \end{align} 
    
    By Lemma \ref{lemma:optimality_condition}, it holds:

    \begin{align}
        \D^T \boldsymbol{W}_{\Psi, \delta}^T\bar{\boldsymbol{g}} = - \frac{2}{\lambda}\K^T (\K\x^*_{\Psi, \delta} - \y^\delta),
    \end{align}

    therefore:

    \begin{align}
    \begin{split}
        \mathcal{B}_{\mathcal{R}_{\Psi, \delta}}(\x', \x^*_{\Psi, \delta}) &=  \frac{2}{\lambda} \langle \K^T (\K\x^*_{\Psi, \delta} - \y^\delta), \x^*_{\Psi, \delta} - \x' \rangle \\&= \frac{2}{\lambda} \langle \K\x^*_{\Psi, \delta} - \y^\delta, \K (\x^*_{\Psi, \delta} - \x') \rangle = 0,
    \end{split}
    \end{align}

    where the last equality follows since $\x^*_{\Psi, \delta} - \x' \in \ker(\K)$ by Lemma \ref{lemma:J_properties_on_M}. Consequently,

    \begin{align}
    \begin{split}
        &0 = \langle \D^T \boldsymbol{W}_{\Psi, \delta}^T \bar{\boldsymbol{g}}, \x^*_{\Psi, \delta} - \x' \rangle = \langle \bar{\boldsymbol{g}}, \boldsymbol{W}_{\Psi, \delta} \D \x^*_{\Psi, \delta}  \rangle - \langle \bar{\boldsymbol{g}}, \boldsymbol{W}_{\Psi, \delta} \D \x' \rangle \\&\iff \mathcal{R}_{\Psi, \delta}(\x') = \langle \bar{\boldsymbol{g}}, \boldsymbol{W}_{\Psi, \delta} \D \x' \rangle,
    \end{split}
    \end{align}

    where we used that:

    \begin{align}
    \begin{split}
        \langle \bar{\boldsymbol{g}}, \boldsymbol{W}_{\Psi, \delta} \D \x^*_{\Psi, \delta} \rangle &= \sum_{i=1}^{2n} (\w(\Psi(\y^\delta)))_i (\bar{\boldsymbol{g}})_i (\D \x^*_{\Psi, \delta})_i \\ 
        &=  \sum_{i \in \{1, \dots, 2n\} \cap \left( \bar{I}(\x^*_{\Psi, \delta}) \right)^c} (\w(\Psi(\y^\delta)))_i \bar{\boldsymbol{g}})_i (\D \x^*_{\Psi, \delta})_i 
        + \sum_{i \in \bar{I}(\x^*_{\Psi, \delta})} (\w(\Psi(\y^\delta)))_i (\bar{\boldsymbol{g}})_i (\D \x^*_{\Psi, \delta})_i \\ 
        &= \sum_{i \in \{1, \dots, n\} \cap \left( \bar{I}(\x^*_{\Psi, \delta}) \right)^c} 
        (\w(\Psi(\y^\delta)))_i \frac{(\D_h \x^*_{\Psi, \delta})_i^2}{(|\D \x^*_{\Psi, \delta}|)_i} + \sum_{i \in \{n+1, \dots, 2n\} \cap \left( \bar{I}(\x^*_{\Psi, \delta}) \right)^c} 
        (\w(\Psi(\y^\delta)))_i \frac{(\D_v \x^*_{\Psi, \delta})_i^2}{(|\D \x^*_{\Psi, \delta}|)_i} \\
        &= \sum_{i \in \{1, \dots, n\}} (\w(\Psi(\y^\delta)))_i 
        \frac{(\D_h \x^*_{\Psi, \delta})_i^2 + (\D_v \x^*_{\Psi, \delta})_i^2}{(|\D \x^*_{\Psi, \delta}|)_i} \\ 
        &= \sum_{i \in \{1, \dots, n\}} (\w(\Psi(\y^\delta)))_i(|\D \x^*_{\Psi, \delta}|)_i \\
        &= \mathcal{R}_{\Psi, \delta}(\x^*_{\Psi, \delta}) = \mathcal{R}_{\Psi, \delta}(\x'),
    \end{split}
    \end{align}

    where the last equality follows by Lemma \ref{lemma:J_properties_on_M}. 

    Since by hypothesis $\x' \notin S(\x^*_{\Psi, \delta})$, there exists at least an index $\hat{i}$ such that $(| \D \x^*_{\Psi, \delta} |)_{\hat{i}} = 0$ but $(| \D \x' |)_{\hat{i}} \neq 0$. Therefore, defining $\bar{\boldsymbol{g}}_h = (\bar{\boldsymbol{g}}_1, \bar{\boldsymbol{g}}_2, \dots, \bar{\boldsymbol{g}}_n) \in \R^n$, $\bar{\boldsymbol{g}}_v = (\bar{\boldsymbol{g}}_{n+1}, \bar{\boldsymbol{g}}_{n+2}, \dots, \bar{\boldsymbol{g}}_{2n}) \in \R^n$, and $| \bar{\boldsymbol{g}} | \in \R^n$ as $(| \bar{\boldsymbol{g}} |)_i = \sqrt{(\bar{\boldsymbol{g}}_h)_i^2 + (\bar{\boldsymbol{g}}_v)_i^2}$, it holds:
    \begin{align}
    \begin{split}
        \mathcal{R}_{\Psi, \delta}(\x') &= \langle \bar{\boldsymbol{g}}, \boldsymbol{W}_{\Psi, \delta} \D \x' \rangle = \sum_{i=1}^{2n} (\w(\Psi(\y^\delta)))_i (\bar{\boldsymbol{g}})_i (\D \x')_i = \langle \begin{bmatrix}\bar{\boldsymbol{g}}_h \\ \bar{\boldsymbol{g}}_v\end{bmatrix}, \boldsymbol{W}_{\Psi, \delta} \begin{bmatrix}\D_h \x' \\ \D_v \x'\end{bmatrix}\rangle \\& \leq \langle | \bar{\boldsymbol{g}} |, | \boldsymbol{W}_{\Psi, \delta} \D \x' | \rangle = \sum_{i=1}^{n} (\w(\Psi(\y^\delta)))_i (| \bar{\boldsymbol{g}} |)_i (| \D \x' |)_i \\ &= \sum_{\substack{i = 1 \\ i \neq \hat{i}}}^{n} (\w(\Psi(\y^\delta)))_i (| \bar{\boldsymbol{g}} |)_i (| \D \x' |)_i + (\w(\Psi(\y^\delta)))_{\hat{i}}|c| (| \D \x' |)_{\hat{i}} \\ & < \sum_{\substack{i = 1 \\ i \neq \hat{i}}}^{n} (\w(\Psi(\y^\delta)))_i(| \D \x' |)_i + (\w(\Psi(\y^\delta)))_{\hat{i}}(| \D \x' |)_{\hat{i}} = \mathcal{R}_{\Psi, \delta}(\x'),
    \end{split}
    \end{align}

    where the three inequalities follow respectively by the Cauchy-Schwartz inequality applied to the vectors $\begin{bmatrix}\bar{\boldsymbol{g}}_h \\ \bar{\boldsymbol{g}}_v\end{bmatrix}$ and $\boldsymbol{W}_{\Psi, \delta} \begin{bmatrix}\D_h \x' \\ \D_v \x'\end{bmatrix}$, the observation that $(| \bar{\boldsymbol{g}} |)_i \leq 1$ for any $i \notin \bar{I}(\x^*_{\Psi, \delta})$, and the assumption that $|c| < 1$. Note that the above observation leads to a contradiction as it shows that $\mathcal{R}_{\Psi, \delta}(\x') < \mathcal{R}_{\Psi, \delta}(\x')$, which implies that $\x'$ has to lie in $S(\x^*_{\Psi, \delta})$. Therefore $\x' = \x^*_{\Psi, \delta}$, proving the uniqueness of the solution.
\end{proof}

The subsequent step involves analyzing the stability of the solution in relation to the noise present in the data.
Indeed, differently from \cite{scherzer2009variational}, where the regularizer is independent of the noise parameter, in our $\Psi$-$W\ell_1$ approach, the regularizer $\mathcal{R}_{\Psi, \delta}$ also depends on $\delta$. To the best of our knowledge, this approach has only been theoretically analyzed in \cite{bianchi2023data}, which, however, addresses a slightly different problem. We provide a full proof following the scheme of \cite[Theorem 3.2]{scherzer2009variational}. 

\begin{lemma}\label{lemma:delta_1_2_inequality}
    For any $\delta_1, \delta_2 \geq 0$, any $\Psi$ and any $\x \in \X$, it holds:
    \begin{align}
        \mathcal{J}_{\Psi, \delta_1}(\x) \leq 2 \mathcal{J}_{\Psi, \delta_2}(\x) + 2 || \y^{\delta_1} - \y^{\delta_2} ||_2^2 + \lambda || \left(\w(\Psi(\y^{\delta_1})) - \w(\Psi(\y^{\delta_2})) \right) ||_1 || \> | \D \x | \> ||_1
    \end{align}
\end{lemma}
\begin{proof}
    The proof is in Appendix \ref{app:delta_1_2_inequality}.
\end{proof}

The following lemma establishes the boundedness of the sequence $\{ \x^*_{\Psi, \delta_k} \}_{k \in \mathbb{N}}$ which consists of the minimizers of 
 $\mathcal{J}_{\Psi, \delta_k}(\x)$.

\begin{lemma}\label{lemma:minimizers_bounded}
    Let $\{ \delta_k \}_{k \in \mathbb{N}}$ be any sequence of noise levels such that $\delta_k \to 0$ as $k \to \infty$. For any $k \in \mathbb{N}$, let $\x^*_{\Psi, \delta_k}$ be the unique minimizer of $\mathcal{J}_{\Psi, \delta_k}(\x)$. Then, the sequence $\{ \x^*_{\Psi, \delta_k} \}_{k \in \mathbb{N}}$ is bounded.
\end{lemma}

\begin{proof}
    The proof is in Appendix \ref{app:minimizers_bounded}.
\end{proof}

\begin{theorem}[Noise Stability]\label{prop:noise_stability}
    Let $\{ \delta_k \}_{k \in \mathbb{N}}$ be any sequence of positive noise levels such that $\delta_k \to 0$ as $k \to \infty$. For any $k \in \mathbb{N}$, let $\x^*_{\Psi, \delta_k}$ be the unique minimizer of $\mathcal{J}_{\Psi, \delta_k}(\x)$. Then $\{ \x^*_{\Psi, \delta_k} \}_{k \in \mathbb{N}}$ has a convergent subsequence, whose limit point corresponds to $\x^*_{\Psi, 0}$, i.e. the unique minimizer of $\mathcal{J}_{\Psi, 0}(\x)$ (i.e $\mathcal{J}_{\Psi, \delta}(\x)$ with $\delta = 0$).
\end{theorem}

\begin{proof}
    Note that, because $\x^*_{\Psi, \delta_k}$ is a minimizer of $\mathcal{J}_{\Psi, \delta_k}(\x)$, for any $\x \in \X$, it holds that $\mathcal{J}_{\Psi, \delta_k}(\x^*_{\Psi, \delta_k}) \leq \mathcal{J}_{\Psi, \delta_k}(\x)$. Let $\bar{\x}$ be any element in $\X$. By applying Lemma \ref{lemma:delta_1_2_inequality} twice, once with $\delta_1 = 0$ and $\delta_2 = \delta_k$, and then with $\delta_1 = \delta_k$ and $\delta_2 = 0$, it follows that
    \begin{align}
    \begin{split}
        \mathcal{J}_{\Psi, 0}(\x^*_{\Psi, \delta_k}) &\leq 2 \mathcal{J}_{\Psi, \delta_k}(\x^*_{\Psi, \delta_k}) + 2 || \y^{\delta_k} - \y^{0} ||_2^2 + \lambda || \w(\Psi(\y^{\delta_k})) - \w(\Psi(\y^{0}))  ||_1 || \> | \D \x^*_{\Psi, \delta_k} | \> ||_1 \\ & \leq  2 \mathcal{J}_{\Psi, \delta_k}(\bar{\x}) + 2 || \y^{\delta_k} - \y^{0} ||_2^2 + \lambda || \w(\Psi(\y^{\delta_k})) - \w(\Psi(\y^{0}))  ||_1 || \> | \D \x^*_{\Psi, \delta_k} | \> ||_1 \\ &\leq 4 \mathcal{J}_{\Psi, 0}(\bar{\x}) + 4 || \y^{\delta_k} - \y^{0} ||_2^2 + \lambda || \w(\Psi(\y^{\delta_k})) - \w(\Psi(\y^{0}))  ||_1 \left(|| \> | \D \x^*_{\Psi, \delta_k} | \> ||_1 + 2|| \> | \D \bar{\x} | \> ||_1 \right).
    \end{split}
    \end{align}

    By the continuity of $\w(\tilde{\x})$ with respect to $\tilde{\x}$ and the continuity of $\Psi(\y^\delta)$ with respect to $\y^\delta$, there exists a function $g(\delta)$ such that $g(\delta) \to 0$ as $\delta \to 0$, such that:
    \begin{align}
        || \w(\Psi(\y^{\delta})) - \w(\Psi(\y^{0}))  ||_1 \leq g(\delta),
    \end{align}
    for any $\delta \geq 0$ sufficiently close to $0$.  Therefore:
    \begin{align}
        \mathcal{J}_{\Psi, 0}(\x^*_{\Psi, \delta_k}) \leq 4 \mathcal{J}_{\Psi, 0}(\bar{\x}) + 4 \delta_k^2 + \lambda g(\delta_k) \left(|| \> | \D \x^*_{\Psi, \delta_k} | \> ||_1 + 2|| \> | \D \bar{\x} | \> ||_1 \right).
    \end{align}

    Note that, by Lemma \ref{lemma:minimizers_bounded}, the sequence $\{ \x^*_{\Psi, \delta_k} \}_{k \in \mathbb{N}}$ is bounded, and in particular, $|| \> | \D \x^*_{\Psi, \delta_k} | \> ||_1 \leq C$ for a constant $C > 0$. Consequently, for any sufficiently small value of $\delta_k$, the sequence $\mathcal{J}_{\Psi, 0}(\x^*_{\Psi, \delta_k})$ is bounded by:
    \begin{align}
        M := 4\mathcal{J}_{\Psi, 0}(\bar{\x}) + 1,
    \end{align}
    which implies that $\{ \x^*_{\Psi, \delta_k} \}_{k \in \mathbb{N}}$ is contained in the $M$-th level set of $\mathcal{J}_{\Psi, 0}$ for any $k \geq k_0$. By compactness of the level sets of $\mathcal{J}_{\Psi, 0}$, it follows that $\{ \x^*_{\Psi, \delta_k} \}_{k \in \mathbb{N}}$ has a convergent subsequence $\{ \x^*_{\Psi, \delta_{k_j}} \}_{j \in \mathbb{N}}$, converging to $\x^* \in \X$. To conclude, note that for any $\x \in \X$:

    \begin{align}
    \begin{split}
        \mathcal{J}_{\Psi, 0}(\x^*) &= || \K \x^* - \y^0 ||_2^2 + \lambda || \w(\Psi(\y^0)) \odot | \D \x^* |\> ||_1 \\ &= \lim_{j\to\infty} || \K \x^*_{\Psi, \delta_{k_j}} - \y^{\delta_{k_j}} ||_2^2 + \lambda || \w(\Psi(\y^{\delta_{k_j}})) \odot | \D \x^*_{\Psi, \delta_{k_j}} |\> ||_1 \\ &\leq \lim_{j\to\infty} || \K \x - \y^{\delta_{k_j}} ||_2^2 + \lambda || \w(\Psi(\y^{\delta_{k_j}})) \odot | \D \x |\> ||_1 \\ &= || \K \x - \y^0 ||_2^2 + \lambda || \w(\Psi(\y^0)) \odot | \D \x |\> ||_1 = \mathcal{J}_{\Psi, 0}(\x),
    \end{split}
    \end{align}

    which implies that $\x^*$ is a minimizer of $\mathcal{J}_{\Psi, 0}(\x^*)$ and $\x^* = \x^*_{\Psi, 0}$ for the uniqueness property, proved in Theorem \ref{prop:unicity}. 
    
\end{proof}

The following result constitutes the primary theoretical contribution of this paper. It demonstrates that if the chosen reconstructor $\Psi$, esponsible for computing the initial guess for the weight, is \emph{sufficiently accurate} (in the sense that it closely approximates a target reconstructor $\Psi^*$), then the solution $\x^*_{\Psi, \delta}$ of the $\Psi-W\ell_1$ will closely approximate the solution obtained using the \emph{ideal} reconstructor 
 $\Psi^*$. This concept will be formalized and further discussed in subsequent sections.

We need two more lemmas to prove the central theorem.
\begin{lemma}\label{lemma:psi_1_2_inequality}
    For any $\Psi_1, \Psi_2$, any $\delta\geq 0$, and any $\x \in \X$, it holds:
    \begin{align}
        \mathcal{J}_{\Psi_1, \delta}(\x) \leq \mathcal{J}_{\Psi_2, \delta}(\x) + \lambda || \left(\w(\Psi_1(\y^{\delta})) - \w(\Psi_2(\y^{\delta})) \right) ||_1 || \> | \D \x | \> ||_1
    \end{align}
\end{lemma}

\begin{proof}
    The proof is in Appendix \ref{app:psi_1_2_inequality}.
\end{proof}

\begin{lemma}\label{lemma:minimizers_bounded_2}
    Let $\{ \Psi_k \}_{k \in \mathbb{N}}$ be any sequence of reconstructors such that $\Psi_k \to \Psi$ as $k \to \infty$ in the sense of Theorem \ref{prop:reconstructor_stability}. For any $k \in \mathbb{N}$, let $\x^*_{\Psi_k, \delta}$ be the unique minimizer of $\mathcal{J}_{\Psi_k, \delta}(\x)$. Then, the sequence $\{ \x^*_{\Psi_k, \delta} \}_{k \in \mathbb{N}}$ is bounded.
\end{lemma}

\begin{proof}
    The proof is in Appendix \ref{app:minimizers_bounded_2}.
\end{proof}

\begin{theorem}[Reconstructor Stability]\label{prop:reconstructor_stability}
    Let $\{ \Psi_k \}_{k \in \mathbb{N}}$ be a sequence of reconstructors such that $\Psi_k \to \Psi^*$ as $k \to \infty$, meaning that $\sup_{\y^\delta \in \mathcal{Y}^\delta} || \Psi_k(\y^\delta) - \Psi^*(\y^\delta) ||_1 \to 0$ as $k \to \infty$, where $\mathcal{Y}^\delta = \{ \y^\delta \in \R^m; \inf_{\x \in \X} || \K\x - \y^\delta ||_2 \leq \delta \}$. Let $\x^*_{\Psi_k, \delta}$ be the unique minimizer of $\mathcal{J}_{\Psi_k, \delta}(\x)$, for a given $\delta \geq 0$. Then $\{ \x^*_{\Psi_k, \delta} \}_{k \in \mathbb{N}}$ has a convergent subsequence, whose limit point is $\x^*_{\Psi^*, \delta}$, i.e. the unique minimizer of $\mathcal{J}_{\Psi^*, \delta}(\x)$. 
\end{theorem}

\begin{proof}
    The proof proceeds as in Theorem \ref{prop:noise_stability}. By applying Lemma \ref{lemma:psi_1_2_inequality} twice we get:
    \begin{align}
    \begin{split}
        \mathcal{J}_{\Psi^*, \delta}(\x^*_{\Psi_k, \delta}) &\leq \mathcal{J}_{\Psi_k, \delta}(\x^*_{\Psi_k, \delta}) + \lambda|| \w(\Psi_k(\y^\delta)) - \w(\Psi^*(\y^\delta) ||_1 || \> | \D\x^*_{\Psi_k, \delta} \ |\> ||_1 \\ &\leq \mathcal{J}_{\Psi_k, \delta}(\bar{\x}) + \lambda|| \w(\Psi_k(\y^\delta)) - \w(\Psi^*(\y^\delta)) ||_1 || \> | \D\x^*_{\Psi_k, \delta} \ |\> ||_1 \\ &\leq \mathcal{J}_{\Psi^*, \delta}(\bar{\x}) + \lambda|| \w(\Psi_k(\y^\delta)) - \w(\Psi^*(\y^\delta)) ||_1 \left(|| \> | \D\bar{\x} \ |\> ||_1 + || \> | \D\x^*_{\Psi_k, \delta} \ |\> ||_1 \right) \\ &\leq \mathcal{J}_{\Psi^*, \delta}(\bar{\x}) + \lambda L_{\w}|| \Psi_k(\y^\delta) - \Psi^*(\y^\delta) ||_1 \left(|| \> | \D\bar{\x} \ |\> ||_1 + || \> | \D\x^*_{\Psi_k, \delta} \ |\> ||_1 \right),
    \end{split}
    \end{align}
    where $L_{\w}$ is the Lipschitz constant of $\w$. Since $\{ \x^*_{\Psi^*, \delta}\}_{k \in \mathbb{N}}$ is bounded due to Lemma \ref{lemma:minimizers_bounded_2}, and since $|| \Psi_k(\y^\delta) - \Psi^*(\y^\delta) ||_1 \to 0$ as $k \to \infty$ by hypothesis, then for $k \geq k_0$, the sequence $\mathcal{J}_{\Psi^*, \delta}(\x^*_{\Psi_k, \delta})$ is bounded by $M$, defined as:
    \begin{align}
        M := 4\mathcal{J}_{\Psi^*, \delta}(\bar{\x}) + 1,
    \end{align}
    which implies that $\{ \x^*_{\Psi_k, \delta} \}_{k \in \mathbb{N}}$ is contained in the $M$-th level set of $\mathcal{J}_{\Psi^*, \delta}$, and in particular it has a convergent subsequence $\{ \x^*_{\Psi_{k_j}, \delta} \}_{j \in \mathbb{N}}$, converging to $\x^* \in \X$. 

    Since, for any $\x \in \X$,
    \begin{align}
    \begin{split}
        \mathcal{J}_{\Psi^*, \delta}(\x^*) &= || \K \x^* - \y^\delta ||_2^2 + \lambda || \w(\Psi^*(\y^\delta)) \odot | \D \x^* |\> ||_1 \\ &= \lim_{j\to\infty} || \K \x^*_{\Psi_{k_j}, \delta} - \y^{\delta} ||_2^2 + \lambda || \w(\Psi_{k_j}(\y^{\delta})) \odot | \D \x^*_{\Psi_{k_j}, \delta} |\> ||_1 \\ &\leq \lim_{j\to\infty} || \K \x - \y^{\delta} ||_2^2 + \lambda || \w(\Psi_{k_j}(\y^{\delta})) \odot | \D \x | \> ||_1 \\ &= || \K \x - \y^\delta ||_2^2 + \lambda || \w(\Psi^*(\y^\delta)) \odot | \D \x |\> ||_1 = \mathcal{J}_{\Psi^*, \delta}(\x),
    \end{split}
    \end{align}
    then $\x^*$ is a minimizers of $\mathcal{J}_{\Psi^*, \delta}(\x)$, i.e. $\x^* = \x^*_{\Psi^*, \delta}$.
\end{proof}

Interestingly, it is possible to formulate Theorem \ref{prop:reconstructor_stability} with alternative hypothesis, shading light to a property of our framework which will be explored in the experimental section. Indeed, since our choice of $\w(\Psi(\y^\delta))$ only depends on the gradient of $\Psi(\y^\delta)$, stability holds if we assume that $||\ |\D \Psi_k (\y^\delta) | - | \D \Psi^* (\y^\delta) | \ ||_1 \to 0$ as $k \to \infty$. In particular:

\begin{corollary}\label{corollary:reconstructor_stability_D}
    Let $\{ \Psi_k \}_{k \in \mathbb{N}}$ be a sequence of reconstructors such that $\Psi_k \to \Psi^*$ as $k \to \infty$, meaning that $\sup_{\y^\delta \in \mathcal{Y}^\delta} || \> | \D \Psi_k (\y^\delta) | - | \D \Psi^* (\y^\delta) | \> ||_1 \to 0$ as $k \to \infty$, where $\mathcal{Y}^\delta = \{ \y^\delta \in \R^m; \inf_{\x \in \X} || \K\x - \y^\delta ||_2 \leq \delta \}$. Let $\x^*_{\Psi_k, \delta}$ be the unique minimizer of $\mathcal{J}_{\Psi_k, \delta}(\x)$, for a given $\delta \geq 0$. Then $\{ \x^*_{\Psi_k, \delta} \}_{k \in \mathbb{N}}$ has a convergent subsequence, whose limit point is $\x^*_{\Psi^*, \delta}$, i.e. the unique minimizer of $\mathcal{J}_{\Psi^*, \delta}(\x)$. 
\end{corollary}

\begin{proof}
    The proof proceeds as in Theorem \ref{prop:reconstructor_stability}, considering the inequality:

    \begin{align}
        || \w(\Psi_k(\y^\delta)) - \w(\Psi^*(\y^\delta)) ||_1 \leq L'_{\w}|| \> | \D \Psi_k(\y^\delta)| - | \D \Psi^*(\y^\delta) | \> ||_1,
    \end{align}
    where $L_{\w}'$ is the Lipschitz constant of $\w(\Psi(\y^\delta))$ with respect to $| \D \Psi(\y^\delta) |$.
\end{proof}

Another useful extension of Theorem \ref{prop:reconstructor_stability} is to consider the situation where $\Psi(\y^\delta) \approx \x^{GT}$ for any $\y^\delta$. Indeed, in Section \ref{ssec:Sintetica} we already observed that the solution $\x^*_{GT, \delta}$, obtained by computing the weights on $\x^{GT}$, approximates well the true solution $\x^{GT}$ even for high noise levels $\delta \geq 0$. Interestingly, Theorem \ref{prop:reconstructor_stability} shows that reconstructor stability holds in this case as well.

\begin{corollary}\label{corollary:reconstructor_stability_gt}
    Let $\y^\delta = \K \x + \e$ be a fixed vector of observations. Let $\{ \Psi_k \}_{k \in \mathbb{N}}$ be a sequence of reconstructors such that one of the following holds:
    \begin{enumerate}
        \item $|| \Psi_k (\y^\delta) - \x^{GT} ||_1 \to 0$ as $k \to \infty$.
        \item $|| \> | \D \Psi_k (\y^\delta) | - | \D \x^{GT} | \> ||_1 \to 0$ as $k \to \infty$.
    \end{enumerate} 
    Let $\x^*_{\Psi_k, \delta}$ be the unique minimizer of $\mathcal{J}_{\Psi_k, \delta}(\x)$, for a given $\delta \geq 0$. Then $\{ \x^*_{\Psi_k, \delta} \}_{k \in \mathbb{N}}$ has a convergent subsequence, whose limit point is $\x^*_{GT, \delta}$, i.e. the unique minimizer of $\mathcal{J}_{GT, \delta}(\x)$. 
\end{corollary}

\begin{proof}
    The proof proceeds as in Theorem \ref{prop:reconstructor_stability} and Corollary \ref{corollary:reconstructor_stability_D}.
\end{proof}

We emphasize that the quantities minimized in assumptions 1 and 2 of the previous corollary are the same, up to the norm, as those minimized in the loss functions \eqref{eq:mseloss} and \ref{eq:gradloss}, respectively.

\section{Numerical Results \label{sec:numres}}

In this section, we validate our proposal with numerical experiments for different tomographic scenarios.

We have built digital simulations based on two different datasets of images, providing ground-truth samples. The first one is the COULE dataset\footnote{\url{https://www.kaggle.com/datasets/loiboresearchgroup/coule-dataset}})
which contains 
256$\times$256 grayscale images, with overlapping geometrical objects having uniform intensities. We used $N_{D} = 400$ images for training, and the test image we used as an exemplary sample is depicted in Figure \ref{fig:GT} with two zoomed crops focusing on objects with very low gray contrasts.
The second dataset comes from the real chest images of the Low Dose CT Grand Challenge, by the Mayo Clinic \cite{mccollough2016tu}. They are downscaled to $256\times 256$ pixel resolution. We used $N_{D} = 3305$ images for training. Figure \ref{fig:GT} also depicts a test image and a close-up of it, in a region containing both high-contrast chest bones, low-contrast soft tissues, and thin details inside the lung.

\begin{figure}[h]
\centering
\captionsetup[subfigure]{}
\subfloat[COULE test image]{
\begin{tikzpicture}
        \node [anchor=south west, inner sep=0] (image) at (0,0) {\includegraphics[width=0.18\linewidth]{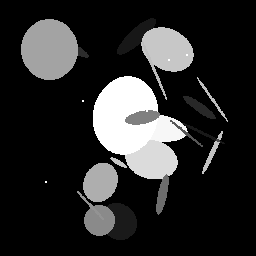}};
        \begin{scope}[x={(image.south east)}, y={(image.north west)}]
            \draw[red, thick] (0.22, 0.58) rectangle (0.52, 0.88);
            \draw[red, thick] (0.54, 0.38) rectangle (0.84, 0.68);
        \end{scope}
    \end{tikzpicture}  
\includegraphics[trim={15mm 38mm 30mm 7mm},clip, width=0.18\linewidth]{imm_Coule/ground_truth.png}
\includegraphics[trim={35mm 25mm 10mm 20mm},clip, width=0.18\linewidth]{imm_Coule/ground_truth.png} 
} \quad
\subfloat[Mayo test image]{
\begin{tikzpicture}
        \node [anchor=south west, inner sep=0] (image) at (0,0) {\includegraphics[width=0.18\linewidth]{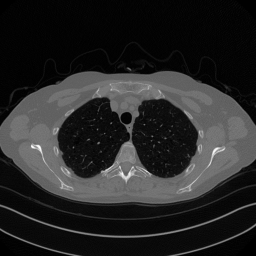}};
        \begin{scope}[x={(image.south east)}, y={(image.north west)}]
            \draw[red, thick] (0.15, 0.39) rectangle (0.45, 0.69);
        \end{scope}
    \end{tikzpicture}  
\includegraphics[trim={10mm 25mm 35mm 20mm},clip, width=0.18\linewidth]{imm_Mayo_005/mayoGT.png}
}
  \caption{Images used as ground truth in the numerical experiments with some zoom-ins remarking regions of interest. On the left, the $\x^{GT}$ coming from the test subset of COULE; on the right, the one from the Mayo Clinic dataset.}
  \label{fig:GT}
\end{figure}

For both datasets, noisy projections, denoted as $ \mathbf{y}_i^\delta $, were generated from the ground truth images $ \mathbf{x}_i^{GT} $, as previously described in Section \ref{ssec:Sintetica} for the synthetic images. In these experiments, a tomographic angular range of 180 degrees was considered for the body scan, along with a limited number $ N_v $ of projection views (usually, an acquisition is considered complete when the interval between the scan angles is less than 0.5 degrees).  The geometric configuration of the CT scanner is subsequently denoted as $ \mathcal{G}_{N_v} $.
All training procedures were conducted over 50 epochs using the Adam optimizer with its default parameters.

In all tests, we applied the methods previously evaluated on the synthetic image in Section \ref{ssec:Sintetica}, namely GT-$W\ell_1$, with $\tilde{\x}$ set to the ground truth image, as well as TV-$W\ell_1$ and FBP-$W\ell_1$, where $\tilde{\x}$ is obtained from the TV and FBP reconstruction methods, respectively.
Depending on the loss function used during network training, in the case of the image loss defined in \eqref{eq:mseloss}, the network output $\tilde{\mathbf{x}} $ is referred to as the \textit{FBP-Net} image, while the final solution is denoted as \textit{FBP-Net-$ W\ell_1 $}. Similarly, when the gradient loss in \eqref{eq:gradloss} is employed, $ \tilde{\mathbf{x}} $ is termed the \textit{FBP-GNet}, and the final solution is referred to as \textit{FBP-GNet-$ W\ell_1 $}.
To compare our weighted total variation (TV) approaches with those in the literature, we consider two well-known methods from the class of Iterative Reweighted strategies for the $ \ell_1 $ norm (IR$ \ell_1 $). 
 In the following, we denote by \textit{IR$\ell_1$-A} the method proposed in \cite{sidky2014cttpv}, based on the updating rule defined in equation \eqref{eq:weights_ir}, with $p=0$ and $\eta=2\cdot 10^{-3}$.
We denote by \textit{IR$\ell_1$-B} a variant of  \textit{IR$\ell_1$-A}, where, inspired by \cite{deng2015ct}, we use the following iterative update rule:
 \begin{equation}\label{eq:weights_ir_B}
   \left(\w(\x^{(k)})\right)_{i} :=  \exp{ \left(- \frac{ \left(  \D_h \x^{(k)} \right)_i^2 +  \left(  \D_v \x^{(k)}  \right)_i^2}{\eta^2} \right)}
 \end{equation}
Here, the superscript $ k $ denotes the $k$-th iteration of the iterative scheme, and $ \eta > 0 $ controls the strength of the diffusive regularization applied during each iteration of the solver. In all experiments, $ \eta $ was set heuristically to $ 6 \times 10^{-3} $.

%
\subsection{Experiments on COULE synthetic dataset}\label{ssec:Coule}

We first look at the results achieved on the synthetic image of the COULE dataset, reported in Figure \ref{fig:GT}. We remark that this image is not part of the training subset.
We have considered different tomographic protocols, given by $N_v=90$ sparse views with medium-high or high noise perturbations on the sinograms ($\nu = 0.03$ and $\nu=0.05$ respectively), or by only $N_v=45$ very sparse acquisitions with moderate noise ($\nu=0.01$).
We remark that the neural networks $\Psi_{\theta}$ have been trained separately, for each tomographic protocol and noise intensity.

All the model and algorithmic parameters have been tuned to minimize the RE metric on the final solutions. 
For the $\mathcal{G}_{90}$ case with $\nu=0.03$, we have set $\lambda=10$, $p=0.3$ and $\eta=2\cdot 10^{-5}$ in our $\Psi$-W$\ell_1$ models, and $\lambda = 2$ for the global TV.
In the iterative reweighted approaches, we set $\lambda=10$ and  $\lambda = 120$ respectively for IR$\ell_1$-A and IR$\ell_1$-B.
In case $\mathcal{G}_{90}$ and $\nu=0.05$, the regularization parameters have been increased to $\lambda=12$ for $\Psi$-W$\ell_1$, $\lambda =3$ for the global TV,  $ \lambda=12$ for IR$\ell_1$-A and $\lambda = 500$ for IR$\ell_1$-B.
For the $\mathcal{G}_{45}$ case with $\nu=0.01$, we have set $\lambda=5$, $p=0.3$ and $\eta=2\cdot 10^{-5}$ in $\Psi$-W$\ell_1$ model, whereas $\lambda = 1.2$ for the global TV,  $\lambda=5$ for IR$\ell_1$-A and $\lambda = 120$ for IR$\ell_1$-B.

Table \ref{tab:Coule_tutto} reports the relative errors computed with respect to the ground truth reference image, for the $\tilde{\x}$ images (when available, in the first column), the corresponding gradient images (when available, in the second column) and for the final solutions $\x^*_{\Psi, \delta}$ (third column). In the last column and only for our $\Psi$-W$\ell_1$ solutions, we report the RE with respect to $\x^*_{GT, \delta}$, to verify that if $\tilde{\x}$ approximates $\x^{GT}$, then $\x^*_{\Psi, \delta} \approx \x^*_{GT,\delta}$  as predicted by Corollary \ref{corollary:reconstructor_stability_gt}.
In this regard, we include the GT-$W\ell_1$ approach in the Table for theoretical reference only, and it should not be conceived as a competitor because it is not feasible for real applications. 
In Figure \ref{fig:Coule_90viste_nl03} we visualize (through zoomed crops) the solutions  obtained by the different methods within the $\mathcal{G}_{90}$ geometry with $\nu=0.03$. \\
First, we look at the RE($\x^*_{\Psi, \delta}, \x^{GT}$) column in the table, and observe that our FBP-Net-$W\ell_1$ and FBP-GNet-$W\ell_1$ always get the minimum errors, sensibly lower than the others. As expected, the global TV regularization is overcome by adaptive models, above all in the case of $\mathcal{G}_{90}$. Interestingly, the IR$\ell_1$ strategies consistently produce errors that do not approach the magnitude of ours, and the images reported exhibit a lower quality in the reconstructions. On the IR$\ell_1$-A image, the low contrast ellipses are slightly visible, whereas the boundaries of the structures are very degraded and noisy on the IR$\ell_1$-B solution. 
Secondly, by examining the values of RE($\tilde{\x}, \x^{GT}$), we observe the high accuracy of the $\tilde{\x}$ images produced by the neural networks, regardless of the loss function used during training. Furthermore, the error relative to the ground truth (GT) image remains low, even in the image gradient domain (see the RE($|\D\tilde{\x}|, |\D\x^{GT}|$) column), suggesting that our adaptive weights effectively account for structures and edges with high precision. Finally, as expected, the GT-$W\ell_1$ image is nearly identical to $\x^{GT}$, and the last column of the table demonstrates that our methods closely approach the solution to which the theoretical model is expected to converge.

\begin{table}[ht]
\caption{Performance results on a COULE image. In the first three columns, the relative error values computed with reference to the exact $\x^{GT}$ image; in the last column the RE values are relative to the output image $\x^*_{GT, \delta}$. }\label{tab:Coule_tutto}
\adjustbox{width=\textwidth}{
\begin{tabular*}{\textwidth}{ll rrrr}
\toprule
& & {\scriptsize RE($\tilde{\x}, \x^{GT}$)} & {\scriptsize  RE($|\D\tilde{\x}|, |\D\x^{GT}|$) } & {\scriptsize  RE($\x^*_{\Psi, \delta}, \x^{GT}$) } & {\scriptsize   RE($\x^*_{\Psi, \delta}, \x^*_{GT, \delta}$) } \\
\midrule
\multirow{ 8}{*}{\shortstack[l]{$\mathcal{G}_{90}$ with \\ $\nu=0.03$}}   
&  global TV &       - & - & 0.0936 & -  \\
& GT-$W\ell_1$ &              0 & 0 & 0.0249 &  0\\
& TV-$W\ell_1$ &          0.1225 & 0.5329 & 0.0742 & 0.0664\\
& FBP-W$\ell_1$ &         0.3993 & 2.2622 & 0.0914 & 0.0847 \\
& FBP-Net-$W\ell_1$ &     0.0807 & 0.3479 & 0.0636 & 0.0552\\
& FBP-GNet-$W\ell_1$  &   0.0958 & 0.2656 & 0.0555 & 0.0463\\
& IR$\ell_1$-A  &             - & - & 0.0723 & -  \\
& IR$\ell_1$-B  &             - & - & 0.1241 & - \\
\midrule
\multirow{ 8}{*}{ \shortstack[l]{$\mathcal{G}_{90}$ with\\ $\nu=0.05$} }  
&  global TV &       - & - & 0.1242 & -  \\
& GT-$W\ell_1$ &              0 & 0 & 0.0359 &  0\\
& TV-$W\ell_1$ &          0.1773 & 0.7777 & 0.1398 & 0.1320\\
& FBP-W$\ell_1$ &         0.6249 & 3.6167 & 0.1834 & 0.1772 \\
& FBP-Net-$W\ell_1$ &     0.1054 & 0.4713 & 0.1043  & 0.0963\\
& FBP-GNet-$W\ell_1$ &   0.1133 & 0.3940 & 0.1001 & 0.0897\\
& IR$\ell_1$-A &             - & - & 0.1139 & -  \\
& IR$\ell_1$-B &             - & - & 0.1523 & - \\
\midrule
\multirow{ 8}{*}{  \shortstack[l]{$\mathcal{G}_{45}$ with\\ $\nu=0.01$} } 
&  global TV &       - & - & 0.0789 & -  \\
& GT-$W\ell_1$ &              0 & 0 & 0.0180 &  0\\
& TV-$W\ell_1$ &          0.1112 & 0.4817 & 0.0847 & 0.0815\\
& FBP-W$\ell_1$ &         0.3486 & 1.5899 & 0.1039 & 0.1012 \\
& FBP-Net-$W\ell_1$ &     0.0904 & 0.3430 & 0.0762 & 0.0727\\
& FBP-GNet-$W\ell_1$ &   0.1309 & 0.3497 & 0.0665 & 0.0626\\
& IR$\ell_1$-A &             - & - & 0.0808 & -  \\
& IR$\ell_1$-B &             - & - & 0.1106 & - \\
\bottomrule
\end{tabular*}
}
\end{table}

\begin{figure}
\centering

\subfloat[\footnotesize	 GT-$W\ell_1$]{
\begin{tabular}{l}  
\includegraphics[trim={15mm 38mm 30mm 7mm},clip, width=0.18\linewidth]{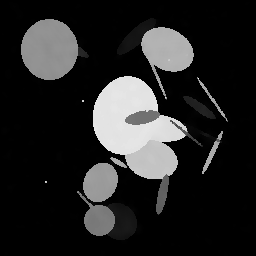}\\
\includegraphics[trim={35mm 25mm 10mm 20mm},clip, width=0.18\linewidth]{imm_Coule/sol_true_sqrt.png}
\end{tabular}
}\hspace{-1.5em}
\subfloat[\footnotesize	 FBP-Net-$W\ell_1$]{
\begin{tabular}{l}
\includegraphics[trim={15mm 38mm 30mm 7mm},clip, width=0.18\linewidth]{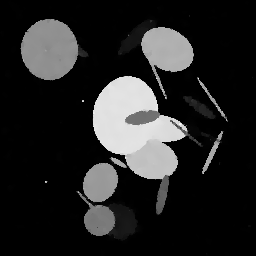}\\
\includegraphics[trim={35mm 25mm 10mm 20mm},clip, width=0.18\linewidth]{imm_Coule/sol_FBP-Net_sqrt.png}
\end{tabular}
}\hspace{-1.5em}
\subfloat[\footnotesize	 FBP-GNet-$W\ell_1$]{
\begin{tabular}{l}
\includegraphics[trim={15mm 38mm 30mm 7mm},clip, width=0.18\linewidth]{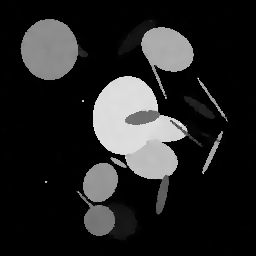}\\
\includegraphics[trim={35mm 25mm 10mm 20mm},clip, width=0.18\linewidth]{imm_Coule/sol_FBP-GNet_sqrt.png}
\end{tabular}
}\hspace{-1.5em}
\subfloat[\footnotesize	 IR$\ell_1$-A]{
\begin{tabular}{l}
\includegraphics[trim={15mm 38mm 30mm 7mm},clip, width=0.18\linewidth]{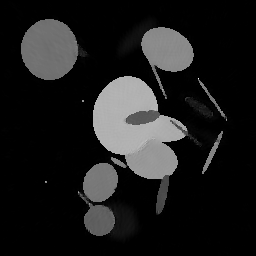}\\
\includegraphics[trim={35mm 25mm 10mm 20mm},clip, width=0.18\linewidth]{imm_Coule/sol_IRCP-TpV_FBP-Net_sqrt.png}
\end{tabular}
}\hspace{-1.5em}
\subfloat[\footnotesize	 IR$\ell_1$-B]{
\begin{tabular}{l}
\includegraphics[trim={15mm 38mm 30mm 7mm},clip, width=0.18\linewidth]{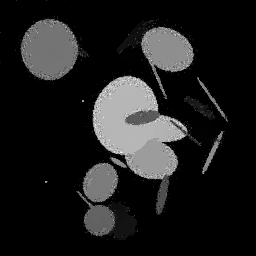}\\ 
\includegraphics[trim={35mm 25mm 10mm 20mm},clip, width=0.18\linewidth]{imm_Coule/sol_IRCP-Exp.png} 
\end{tabular}
}
  \caption{From left to right: crops of the solution images computed (in case $\mathcal{G}_{90}$ and noise with $\nu=0.03$) by our GT-$W\ell_1$, FBP-Net-$W\ell_1$, FBP-GNet-$W\ell_1$ and by the state-of-the-art methods IR$\ell_1$-A and IR$\ell_1$-B. The depicted zooms correspond to those reported in Figure \ref{fig:GT}. }
  \label{fig:Coule_90viste_nl03}
\end{figure}

\subsection{Experiments on Mayo real medical image dataset}\label{ssec:Mayo}

\begin{table}[b]
\caption{Performance results on the real image with $\mathcal{G}_{45}$ geometry and $\nu=0.005$. In the first three columns the metrics refer to the $\tilde{\x}$ image, in the last three columns they refer to the output image $\x^*_{\Psi,\delta}$. All the metrics are relative to the ground truth image.}\label{tab:Mayo005}%
\begin{tabular*}{\textwidth}{@{\extracolsep\fill}l rrr rrr} 
\toprule
 & \multicolumn{3}{@{}c@{}}{ $\tilde{\x} $} & \multicolumn{3}{@{}c@{}}{$\x^*_{\Psi,\delta}$}\\
 & RE & PSNR   & SSIM & RE & PSNR   & SSIM\\
\cmidrule{2-4} \cmidrule{5-7}
\midrule
global TV  & - & - & - & 0.1196 & 30.7946 & 0.8591 \\
FBP-$W\ell_1$                     & 0.2879 & 23.1629 & 0.4133         & 0.1009 & 32.2698 & 0.8832  \\
FBP-Net-$W\ell_1$     & 0.1038 & 32.0253 & 0.8588         & 0.0893 & 33.3329 & 0.9009  \\ 
FBP-GNet-$W\ell_1$     & 0.5092 & 18.2113 & 0.3843         & 0.0834 & 33.9224 & 0.9128  \\   
 IR$\ell_1$-A & - & - & - & 0.1293 & 30.1179 & 0.8274 \\
 IR$\ell_1$-B & - & - & - & 0.1027 & 32.1152 & 0.8717 \\
\bottomrule
\end{tabular*}
\end{table}

We now present the numerical results obtained from the analysis of the real chest tomographic image shown in Figure \ref{fig:GT}. We highlight that this image is not included in the training subset. 
As visible, low-contrast regions and high-contrast tiny details in the lungs characterize the image.
Consequently, the gradient image is less sparse compared to the previous synthetic case. The objective of the following experiments is twofold: to evaluate the performance of the proposed methods on a real image and to examine the behavior of the deep network-based reconstructions under varying levels of noise in the sinograms.\\
We initially consider 45 projections within the angular range $[0, 180]$ degrees, introducing noise to the measurements with a noise level of $\nu = 0.005$.
We set $\lambda = 0.8$, $p=0.3$ and $\eta=2\cdot 10^{-3}$ in our approach and for the global TV. We also set $\lambda = 0.05$ for the IR$\ell_1$-A algorithm and $\lambda = 50$ for IR$\ell_1$-B.

In Table \ref{tab:Mayo005} we report the image quality assessment metrics for both the $\tilde{\x}$ and the solution $\x^*_{\Psi, \delta}$. 
All the metrics reflect that our proposal overcomes the widely used IR strategy of adaptive regularization.
It is also evident how the embedding of a deep neural network improves the final reconstruction quality, with respect to the global TV regularization and to the mere FBP operators within our $\Psi$-$W\ell_1$ framework. \\
The benefit of training the reconstructor with the gradient loss is noticeable (final $RE=0.0834$)  even if the corresponding $\tilde{\x}$ image is not an accurate approximation of $\x^{GT}$ (its relative error is $RE=0.5092$), as theoretically predicted in Corollary  \ref{corollary:reconstructor_stability_gt}.
We have also computed the RE between the gradient images, i.e., $RE(|\D\tilde{\x}|, |\D\x^{GT}|)$, for the three considered $\Psi$-$W\ell_1$ cases.  The results are $1.1687$ for FBP-$W\ell_1$, $0.4309$ for FBP-Net-$W\ell_1$ and $0.3298$ for FBP-GNet-$W\ell_1$, confirming that the gradient image is best approximated by the neural network trained with the gradient loss.
The benefit of using the gradient loss is noticeable   in the left plot of Figure \ref{fig:MayoPlots}, where we compare the relative errors of the FBP-$W\ell_1$, FBP-Net-$W\ell_1$ and FBP-GNet-$W\ell_1$ approaches over the iterations.

\begin{figure}
\centering
\includegraphics[width=0.4\linewidth]{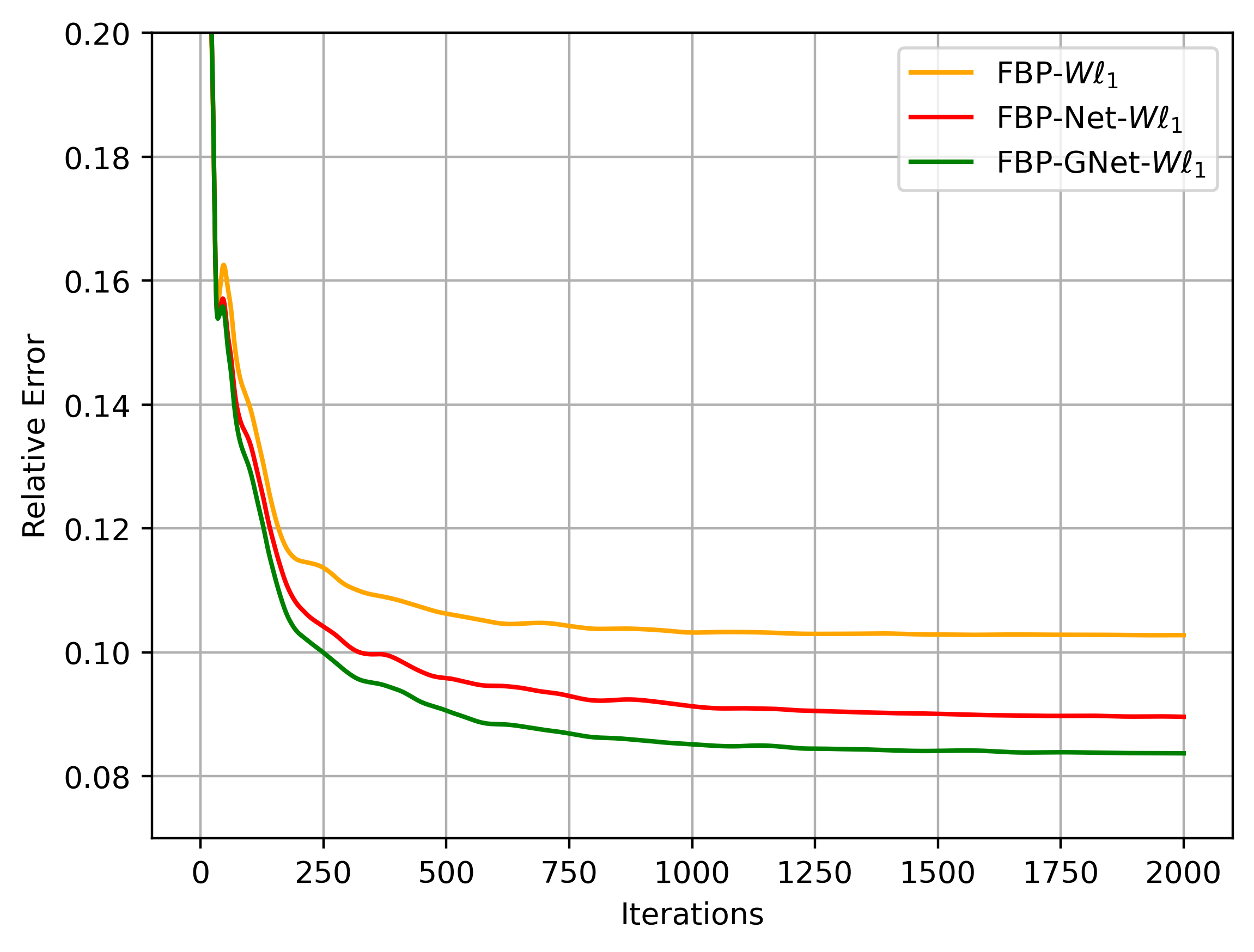}
\includegraphics[width=0.4\linewidth]{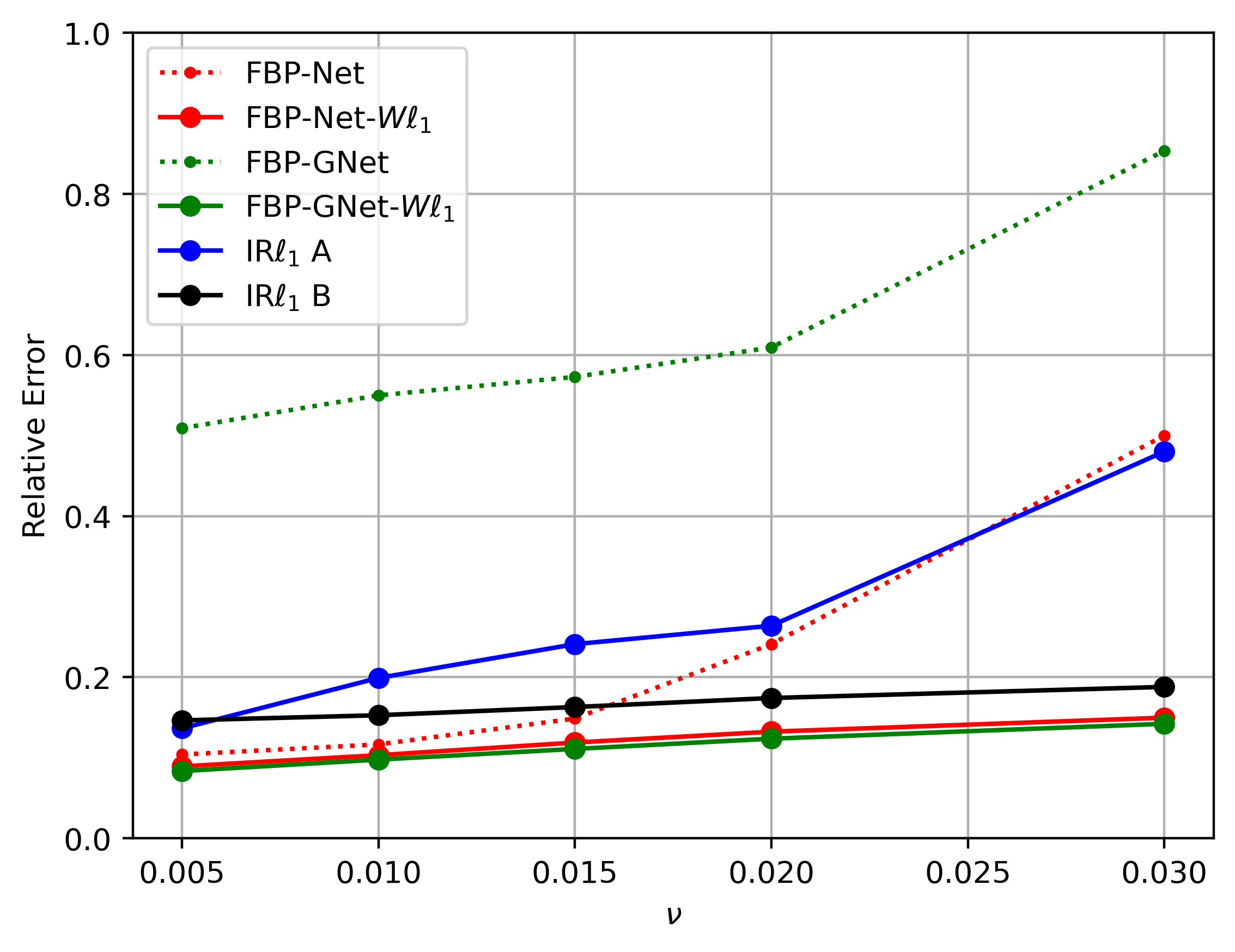}
  \caption{Plots of the relative errors of the solutions computed for the Mayo test image depicted in Figure \ref{fig:GT}. In both cases, the networks have been trained on sinograms affected by noise with $\nu=0.005$ for the $\mathcal{G}_{45}$ geometry. On the left: the RE plot over iterations for the solutions by FBP-$W\ell_1$, FBP-Net-$W\ell_1$ and FBP-GNet-$W\ell_1$.
  On the right: plot of the RE metric for $\nu$ in the range $[0.005, 0.030]$.}
  \label{fig:MayoPlots}
\end{figure}

We extend the evaluation of our FBP-Net-W $\ell_1$ and FBP-GNet-W $\ell_1$ schemes by examining their performance when processing data affected by unexpected noise levels. Specifically, we utilize networks previously trained with $\nu = 0.005$ for G45 and test our approaches on data with $\nu$ values ranging from 0.005 to 0.030. To optimize the results, the regularization parameter $\lambda$ is adjusted to minimize the relative error (RE). This out-of-distribution testing is critical for assessing the reliability and robustness of the models.
Noise injection is often described as an adversarial statistical attack on a network \cite{stabilizing_deep_tomographic_reconstruction_A,stabilizing_deep_tomographic_reconstruction_B}, a practice widely adopted in the deep learning imaging community. In our previous investigations, we have highlighted the susceptibility of learned operators to unseen noise and demonstrated the potential risks associated with employing networks as black-box tools.
 \cite{green_post_processing, evangelista2023ambiguity, loli2023ctprintnet}. 

In the proposed scheme, the networks do not generate the final solution; instead, any potential hallucinations are mitigated by the subsequent variational model. This is demonstrated in the second image of Figure \ref{fig:MayoPlots}, which shows the trends in relative errors for images reconstructed by our methods, as well as by the $W\ell_1$ and IR$\ell_1$ under simulated noise levels of $\nu=0.010$, $\nu=0.015$, $\nu=0.020$ and $\nu=0.030$. As illustrated, the accuracy of outputs from the FBP-Net and FBP-GNet networks declines significantly (depicted by the red and green dotted lines, respectively), highlighting the instability of end-to-end neural networks and corroborating findings reported in previous studies \cite{green_post_processing, evangelista2023ambiguity}.
  However, the corresponding final reconstructions (shown by the red and green solid lines) do not degrade and instead exhibit the stability properties characteristic of regularized variational approaches, as observed in \cite{loli2023ctprintnet}. 

Finally, we examine several images presented in Figure \ref{fig:results_Mayo005}.

In Figure \ref{fig:results_Mayo005}, the first column displays the outputs of the FBP-Net, while the second column shows the resulting FBP-Net-$W\ell_1$ solutions. The third and fourth columns contain the outputs of the FBP-GNet and FBP-GNet-$W\ell_1$, respectively. The images in the first row are generated from simulations with $\nu=0.005$, corresponding to tests statistically consistent with the training samples. In contrast, the images in the last row are derived from simulations with higher unseen noise, characterized by $\nu=0.015$.\\
both networks exhibit difficulty in preserving structural details, with the presence of higher unseen noise significantly degrading output quality. This results in blurred and incomplete features, as well as the appearance of artifacted structures.
In contrast, the final reconstructions demonstrate significant improvements in contrast and edge sharpness, irrespective of the noise level. This outcome underscores the effectiveness of the proposed adaptive weighted model.
Finally, we highlight that FBP-GNet-$W\ell_1$ outperforms FBP-Net-$W\ell_1$, fewer noise artifacts and providing a clearer representation of fine structures, such as those within the lung, even under unseen noise levels. This demonstrates the robustness of the GNet-based initialization and its capacity to generalize effectively to out-of-distribution noise conditions.

\begin{figure}
    \centering
    \subfloat[FBP-Net]{
    \begin{tabular}{c}
         \includegraphics[width=0.18\linewidth]{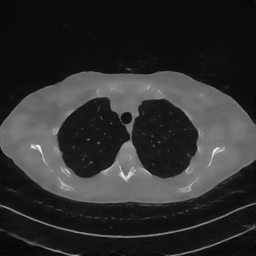}\\
         \includegraphics[trim={10mm 25mm 35mm 20mm},clip, width=0.18\linewidth]{imm_Mayo_005/xtilde_FBP-Net.png}\\
         \\
         \includegraphics[width=0.18\linewidth]{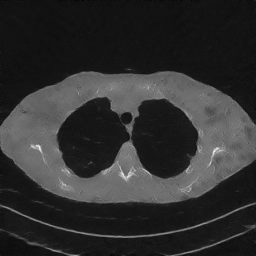}\\
         \includegraphics[trim={10mm 25mm 35mm 20mm},clip, width=0.18\linewidth]{imm_Mayo_005/xtilde_FBP-Net_0.015.png} 
    \end{tabular}
    }
    \subfloat[FBP-Net-$W\ell_1$]{
    \begin{tabular}{c}
         \includegraphics[width=0.18\linewidth]{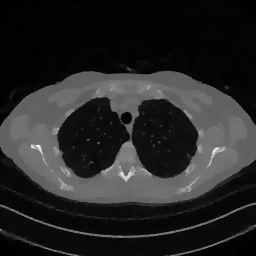}\\
         \includegraphics[trim={10mm 25mm 35mm 20mm},clip, width=0.18\linewidth]{imm_Mayo_005/sol_FBP-Net.png}\\
         \\
         \includegraphics[width=0.18\linewidth]{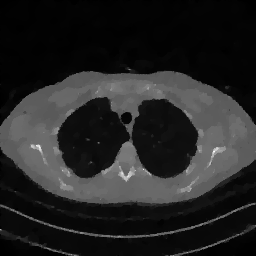}\\
         \includegraphics[trim={10mm 25mm 35mm 20mm},clip, width=0.18\linewidth]{imm_Mayo_005/sol_FBP-Net_0.015.png} 
    \end{tabular}
    } 
    \subfloat[FBP-GNet]{
    \begin{tabular}{c}
         \includegraphics[width=0.18\linewidth]{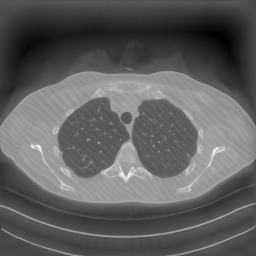}\\
         \includegraphics[trim={10mm 25mm 35mm 20mm},clip, width=0.18\linewidth]{imm_Mayo_005/xtilde_FBP-Net.png}\\
         \\
         \includegraphics[width=0.18\linewidth]{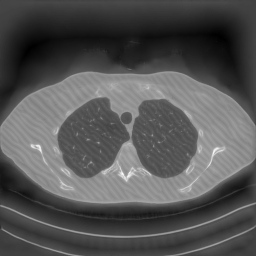}\\
         \includegraphics[trim={10mm 25mm 35mm 20mm},clip, width=0.18\linewidth]{imm_Mayo_005/xtilde_FBP-GNet_0.015.png} 
    \end{tabular}
    }
    \subfloat[FBP-GNet-$W\ell_1$]{
    \begin{tabular}{c}
         \includegraphics[width=0.18\linewidth]{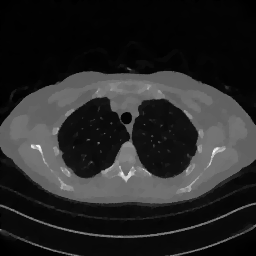}\\
         \includegraphics[trim={10mm 25mm 35mm 20mm},clip, width=0.18\linewidth]{imm_Mayo_005/sol_FBP-GNet.png}\\
         \\
         \includegraphics[width=0.18\linewidth]{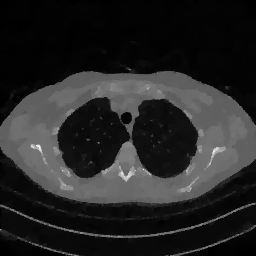}\\
         \includegraphics[trim={10mm 25mm 35mm 20mm},clip, width=0.18\linewidth]{imm_Mayo_005/sol_FBP-GNet_0.015.png} 
    \end{tabular}
    }
    \caption{Results on the Mayo real medical image for the $\mathcal{G}_{45}$ geometry, under different noise levels. Columns show FBP-Net (1st), FBP-Net-$W\ell_1$ (2nd), FBP-GNet (3rd), and FBP-GNet-$W\ell_1$ (4th). The first two rows use $\nu = 0.005$ (training-consistent noise), while the last two rows use higher unseen noise with $\nu = 0.015$. }
  \label{fig:results_Mayo005}
\end{figure}

\section{Conclusions \label{sec:concl}}

In this paper, we proposed an adaptive weighted Total Variation (TV) scheme for reconstructing CT images from few views. The adaptive weights are determined based on an intermediate image, which is expected to approximate the ground truth as closely as possible. We propose computing that intermediate image using a neural network that recovers the ground truth image, or its gradient, from a fast, coarse Filtered Back Projection reconstruction computed upon the sparse sinogram. The results of simulations conducted considering both synthetic and real datasets demonstrate that the proposed approach outperforms different state-of-the-art methods.\\
In conclusion, the proposed framework effectively balances noise reduction and detail preservation by introducing spatially adaptive regularization weights computed from neural network-generated approximations. In addition, theoretical analyses confirmed the well-posedness of the method, ensuring convergence and stability.

Future work may explore expanding the neural network's generalization across diverse tomographic setups and noise conditions, further integrating hybrid optimization and learning approaches to improve computational efficiency and reconstruction fidelity. In particular, given the limited availability of ground truth data in the context of tomography, we will consider ground truth images obtained by means of unsupervised or semi-supervised frameworks such as \cite{evangelista2023rising, ongie2020deep}. In addition, further extensions of this strategy to different sparsity regularizers will be explored, involving different image transforms such as wavelets or shearlets. 

\section*{References}
\bibliographystyle{iopart-num}
\bibliography{biblio}

\providecommand{\newblock}{}
\begin{thebibliography}{10}
\expandafter\ifx\csname url\endcsname\relax
  \def\url#1{{\tt #1}}\fi
\expandafter\ifx\csname urlprefix\endcsname\relax\def\urlprefix{URL }\fi
\providecommand{\eprint}[2][]{\url{#2}}

\bibitem{rudin1992nonlinear}
Rudin L~I, Osher S and Fatemi E 1992 {\em Physica D: nonlinear phenomena\/} {\bf 60} 259--268

\bibitem{sidky2014cttpv}
Sidky E~Y and et~al 2014 {\em IEEE Journal of Translational Engineering in Health and Medicine\/}

\bibitem{piccolomini2021model}
Piccolomini E~L and Morotti E 2021 {\em Journal of Imaging\/} {\bf 7}

\bibitem{friot2022iterative}
Friot L, Peyrin F, Maxim V {\em et~al.\/} 2022 {\em Physics in Medicine \& Biology\/} {\bf 67} 205010

\bibitem{chan2020two}
Chan R~H, Kan K~K, Nikolova M and Plemmons R~J 2020 {\em Journal of Mathematical Imaging and Vision\/} {\bf 62} 790--807

\bibitem{pragliola2023and}
Pragliola M, Calatroni L, Lanza A and Sgallari F 2023 {\em SIAM Review\/} {\bf 65} 601--685

\bibitem{dong2011automated}
Dong Y, Hinterm{\"u}ller M and Rincon-Camacho M~M 2011 {\em Journal of Mathematical Imaging and Vision\/} {\bf 40} 82--104

\bibitem{hintermuller2017optimalII}
Hinterm{\"u}ller M, Rautenberg C~N, Wu T and Langer A 2017 {\em Journal of Mathematical Imaging and Vision\/} {\bf 59} 515--533

\bibitem{bortolotti2016uniform}
Bortolotti V, Brown R, Fantazzini P, Landi G and Zama F 2016 {\em Inverse Problems\/} {\bf 33} 015003

\bibitem{cascarano2023constrained}
Cascarano P, Franchini G, Kobler E, Porta F and Sebastiani A 2023 {\em Computational Optimization and Applications\/} {\bf 84} 125--149

\bibitem{kan2021pnkh}
Kan K, Fung S~W and Ruthotto L 2021 {\em SIAM Journal on Scientific Computing\/} {\bf 43} S704--S726

\bibitem{grasmair2009locally}
Grasmair M 2009 Locally adaptive total variation regularization {\em International Conference on Scale Space and Variational Methods in Computer Vision\/} (Springer) pp 331--342

\bibitem{chen2010adaptive}
Chen Q, Montesinos P, Sun Q~S, Heng P~A {\em et~al.\/} 2010 {\em Image and vision computing\/} {\bf 28} 298--306

\bibitem{bubba2022bilevel}
Bubba T~A, Calatroni L, Catozzi A, Crisci S, Pock T, Pragliola M, Rautio S, Riccio D and Sebastiani A 2022 Bilevel learning of regularization models and their discretization for image deblurring and super-resolution {\em INdAM Workshop: Advanced Techniques in Optimization for Machine learning and Imaging\/} (Springer) pp 55--81

\bibitem{cuomo2023speckle}
Cuomo S, De~Rosa M, Izzo S, Piccialli F and Pragliola M 2023 {\em Applied Numerical Mathematics\/}

\bibitem{kofler2023learning}
Kofler A, Altekr{\"u}ger F, Antarou~Ba F, Kolbitsch C, Papoutsellis E, Schote D, Sirotenko C, Zimmermann F~F and Papafitsoros K 2023 {\em SIAM Journal on Imaging Sciences\/} {\bf 16} 2202--2246

\bibitem{pourya2024iteratively}
Pourya M, Neumayer S and Unser M 2024 {\em Numerical Functional Analysis and Optimization\/} {\bf 45} 411--440

\bibitem{huang2018scale}
Huang Y, Taubmann O, Huang X, Haase V, Lauritsch G and Maier A 2018 {\em IEEE Transactions on Radiation and Plasma Medical Sciences\/} {\bf 2} 307--314

\bibitem{xi2023adaptive}
Xi Y, Zhou P, Yu H, Zhang T, Zhang L, Qiao Z and Liu F 2023 {\em Medical Physics\/} {\bf 50} 5568--5584

\bibitem{luo2018adaptive}
Luo F, Li W, Tu W and Wu W 2018 {\em IEEE Access\/} {\bf 6} 64225--64236

\bibitem{candes2008enhancing}
Cand\`es E~J, Wakin M~B and Boyd S~P 2008 {\em Journal of Fourier analysis and applications\/}

\bibitem{daubechies2010iteratively}
Daubechies I, DeVore R, Fornasier M and G{\"u}nt{\"u}rk C~S 2010 {\em Communications on Pure and Applied Mathematics: A Journal Issued by the Courant Institute of Mathematical Sciences\/} {\bf 63} 1--38

\bibitem{lazzaro2019nonconvex}
Lazzaro D, Piccolomini E~L and Zama F 2019 {\em Inverse Problems\/} {\bf 35} 084002

\bibitem{scherzer2009variational}
Scherzer O, Grasmair M, Grossauer H, Haltmeier M and Lenzen F 2009 {\em Variational methods in imaging\/} vol 167 (Springer)

\bibitem{evangelista2022or}
Evangelista D, Nagy J, Morotti E and Piccolomini E~L 2022 {\em arXiv preprint arXiv:2211.13692\/}

\bibitem{chambolle2011first}
Chambolle A and Pock T 2011 {\em Journal of mathematical imaging and vision\/} {\bf 40} 120--145

\bibitem{bauschke2017correction}
Bauschke H~H, Combettes P~L, Bauschke H~H and Combettes P~L 2017 {\em Convex Analysis and Monotone Operator Theory in Hilbert Spaces\/} (Springer)

\bibitem{kak2001principles}
Kak A~C and Slaney M 2001 {\em Principles of computerized tomographic imaging\/} (SIAM)

\bibitem{wang2004image}
Wang Z, Bovik A~C, Sheikh H~R and Simoncelli E~P 2004 {\em IEEE transactions on image processing\/} {\bf 13} 600--612

\bibitem{green_post_processing}
Morotti E, Evangelista D and Loli~Piccolomini E 2021 {\em Journal of Imaging\/} {\bf 7} 139

\bibitem{evangelista2023ambiguity}
Evangelista D, Morotti E, Piccolomini E~L and Nagy J 2023 {\em Journal of Imaging\/} {\bf 9} ISSN 2313-433X \urlprefix\url{https://www.mdpi.com/2313-433X/9/7/133}

\bibitem{bertero2021introduction}
Bertero M, Boccacci P and De~Mol C 2021 {\em Introduction to inverse problems in imaging\/} (CRC press)

\bibitem{jorgensen2015testable}
J{\o}rgensen J~S, Kruschel C and Lorenz D~A 2015 {\em Inverse Problems in Science and Engineering\/} {\bf 23} 1283--1305

\bibitem{bianchi2023data}
Bianchi D, Evangelista D, Aleotti S, Donatelli M, Piccolomini E~L and Li W 2023 {\em arXiv preprint arXiv:2312.16936\/}

\bibitem{mccollough2016tu}
McCollough C 2016 {\em Medical physics\/} {\bf 43} 3759--3760

\bibitem{deng2015ct}
Deng L, Mi D, He P, Feng P, Yu P, Chen M, Li Z, Wang J and Wei B 2015 {\em Bio-Medical Materials and Engineering\/} {\bf 26} S1685--S1693

\bibitem{stabilizing_deep_tomographic_reconstruction_A}
Wu W, Hu D, Cong W, Shan H, Wang S, Niu C, Yan P, Yu H, Vardhanabhuti V and Wang G 2022 {\em Patterns\/} {\bf 3} 100474

\bibitem{stabilizing_deep_tomographic_reconstruction_B}
Wu W, Hu D, Cong W, Shan H, Wang S, Niu C, Yan P, Yu H, Vardhanabhuti V and Wang G 2022 {\em Patterns\/} {\bf 3} 100475

\bibitem{loli2023ctprintnet}
Loli~Piccolomini E, Prato M, Scipione M and Sebastiani A 2023 {\em Algorithms\/} {\bf 16} 270

\bibitem{evangelista2023rising}
Evangelista D, Morotti E and Piccolomini E~L 2023 {\em Computerized Medical Imaging and Graphics\/} {\bf 103} 102156

\bibitem{ongie2020deep}
Ongie G, Jalal A, Metzler C~A, Baraniuk R~G, Dimakis A~G and Willett R 2020 {\em IEEE Journal on Selected Areas in Information Theory\/} {\bf 1} 39--56

\bibitem{beck2017first}
Beck A 2017 {\em First-order methods in optimization\/} (SIAM)

\end{thebibliography}

\section*{Acknowledgement}
E. Loli Piccolomini, D. Evangelista and E. Morotti are supported by the ``Fondo per il Programma Nazionale di Ricerca e Progetti di Rilevante Interesse Nazionale (PRIN)'' 2022 project ``STILE: Sustainable Tomographic Imaging with Learning and rEgularization'', project  code: 20225STXSB, funded by the European Commission under the NextGeneration EU programme, project code MUR 20225STXSB, CUP J53D23003600006. \\
This work is partially supported by the Gruppo Nazionale per il Calcolo Scientifico (GNCS-INdAM) within the projects "Deep Variational Learning: un approccio combinato per la ricostruzione di immagini" and "MOdelli e MEtodi Numerici per il Trattamento delle Immagini (MOMENTI)", projects code: CUP E53C23001670001
A. Sebastiani is supported by the project “PNRR - Missione 4 “Istruzione e Ricerca” - Componente C2 Investimento 1.1 , PRIN “Advanced optimization METhods for automated central veIn Sign detection in multiple sclerosis from magneTic resonAnce imaging (AMETISTA)”, project code: P2022J9SNP, MUR D.D. financing decree n. 1379 of 1st September 2023 (CUP E53D23017980001) funded by the European Commission under the NextGeneration EU programme.

\begin{appendices}
\section{Proof of well-posedness of \texorpdfstring{$\Psi$-W$\ell_1$}{Psi-WL1} scheme}\label{app:proofs}
In this Section we prove some  lemmas stated  in Section \ref{sec:wellposedness}. 


    
    \subsection{Proof of Lemma \ref{lemma:coercivity}}
    \begin{proof}
        Let $\{ \x_k\}_{k \in \mathbb{N}} \subseteq \X$ any sequence such that $|| \x_k ||_2 \to \infty$ as $k \to \infty$. In particular, $||\x_k||_2 \geq 0$ for any $k \geq \bar{k}$. When it happens, Assumption \ref{assumption:kernels} implies that at least one of the following must hold:
        \begin{enumerate}
            \item $\x_k \in \ker(\K)^c$, which implies that $|| \K \x_k - \y^\delta ||_2^2 \geq f_1(|| \x_k ||_2)$, where $f_1(|| \x_k ||_2) \to \infty$ as $k \to \infty$,
            \item $\x_k \in \ker(\boldsymbol{W}_{\Psi, \delta} \D)^c$, which implies that $\mathcal{R}_{\Psi, \delta}(\x_k) \geq f_2(|| \x_k ||_2)$, where $f_2(|| \x_k ||_2) \to \infty$ as $k \to \infty$,
        \end{enumerate}
        as both $|| \K\x - \y^\delta ||_2^2$ and $\mathcal{R}_{\Psi, \delta}(\x)$ are coercive on $\ker(\K)^c$ and $\ker(\boldsymbol{W}_{\Psi, \delta}\D)^c$, respectively. Since $\mathcal{J}_{\Psi, \delta}(\x_k) \geq \min \{ || \K \x_k - \y^\delta ||_2^2, \mathcal{R}_{\Psi, \delta}(\x_k)\}$, it implies that:
        \begin{align}
            \mathcal{J}_{\Psi, \delta}(\x_k) \geq \min \{ f_1( || \x_k ||_2), f_2( || \x_k ||_2)\} \to \infty, \quad k \to \infty,
        \end{align}
        concluding the proof.
    \end{proof}




\subsection{Proof of Lemma \ref{lemma:TV_gradient}}\label{app:TV_gradient}

\begin{proof}
    The Total Variation can be rewritten as
    $TV(\x) = \sum_{i=1}^n\|\boldsymbol{U}_i \D\x \|_{2} $
    where the linear operator $\boldsymbol{U}_i\in\R^{2\times2n}$ is such that $\boldsymbol{U}_i\D\x=((\D_h\x)_i,(\D_v\x)_i)$ $\forall\ i=1,\ldots,n$.
    Recalling the computation properties of the subgradients \cite{beck2017first}, we derive
     \begin{equation*}    
        \partial TV(\x) = \sum_{i=1}^n\partial (\|\boldsymbol{U}_i\D\x \|_{2}).
     \end{equation*}
    Noticing that $\|\boldsymbol{U}_i\D\x\|_2=(\|\D\x\|)_i$ and focusing on a single term, we obtain
    \begin{equation*}
        \partial (\|\boldsymbol{U}_i\D\x \|_{2}) = \D^T\boldsymbol{U}_i^T \boldsymbol{v},
    \end{equation*}
    where $\boldsymbol{v}\in\R^2$ is defined as follows
    \begin{equation*}
        \begin{cases}
            \frac{\boldsymbol{U}_i\D\x}{| \D \x |_i} & \mbox{ if } (| \D \x |)_i \neq0,\\
            (c,c) & \mbox{ if } (| \D \x |)_i =0,
        \end{cases}
    \end{equation*}
    with $c\in[-1,1]$.
    This concludes the proof since $\boldsymbol{w}=\boldsymbol{U}_i^T\boldsymbol{v}\in\R^{2n}$ is such that
    \begin{equation*}
        (\boldsymbol{w})_j = 
        \begin{cases}
            (\boldsymbol{v})_1 & \mbox{if } j=i,\\
            (\boldsymbol{v})_2 & \mbox{if } j=n+1,\\
            0 & \mbox{otherwise}.
        \end{cases}
    \end{equation*}
\end{proof}

\subsection{Proof of Lemma \ref{lemma:J_properties_on_M}}\label{app:J_properties_on_M}

\begin{proof}
    By simple algebraic manipulations, it is not hard to show that for any $\x_1, \x_2 \in \R^n$, it holds:
    \begin{align}
        \mathcal{J}_{\Psi, \delta}\left(\frac{\x_1 + \x_2}{2} \right) \leq \frac{1}{2} \left( \mathcal{J}_{\Psi, \delta}(\x_1) + \mathcal{J}_{\Psi, \delta}(\x_2) \right) - \frac{1}{8} || \K\x_1 - \K\x_2 ||_2^2.
    \end{align}
    
    Now, let $\x_1, \x_2 \in \mathcal{M}$. By convexity of $\mathcal{J}_{\Psi, \delta}(\x)$ and of $\X$, necessarily $\mathcal{J}_{\Psi, \delta}(\x_1) = \mathcal{J}_{\Psi, \delta}(\x_2) = \mathcal{J}_{\Psi, \delta} (\frac{\x_1 + \x_2}{2}) = \mathcal{J}_{\Psi, \delta}^*$. Therefore:
    \begin{align}
    \begin{split}
        &\mathcal{J}_{\Psi, \delta}^* \leq \frac{1}{2} (\mathcal{J}_{\Psi, \delta}^* + \mathcal{J}_{\Psi, \delta}^*) - \frac{1}{8} || \K\x_1 - \K\x_2 ||_2^2 \\ \iff & \mathcal{J}_{\Psi, \delta}^* \leq \mathcal{J}_{\Psi, \delta}^* - \frac{1}{8} || \K\x_1 - \K\x_2 ||_2^2 \\ \iff & || \K\x_1 - \K\x_2 ||_2^2 \leq 0 \\ \iff & \K\x_1 = \K\x_2,
    \end{split}
    \end{align}
    which proves that $\x_1 - \x_2 \in \ker (\K)$.

    Note that this causes $|| \K\x_1 - \y^\delta ||_2^2$ to be equal to $|| \K\x_2 - \y^\delta||_2^2$ which, together with the observation that $\mathcal{J}_{\Psi, \delta}(\x_1) = \mathcal{J}_{\Psi, \delta}(\x_2)$, implies $\mathcal{R}_{\Psi, \delta}(\x_1) = \mathcal{R}(\x_2)$.
\end{proof}

\subsection{Proof of Lemma \ref{lemma:delta_1_2_inequality}}\label{app:delta_1_2_inequality}

\begin{proof}
    It is known (see e.g. \cite{scherzer2009variational}) that for any $p \geq 1$, and for any $\x_1, \x_2 \in \X$, it holds:
    \begin{align}\label{eq:norm_inequality}
        || \x_1 + \x_2 ||_p^p \leq 2^{p-1} \left( || \x_1 ||_p^p + || \x_2 ||_p^p \right).
    \end{align}

    Applying this inequality on $|| \K\x - \y^{\delta_1} ||_2^2 = || (\K \x - \y^{\delta_2}) + (\y^{\delta_2} - \y^{\delta_1}) ||_2^2$ with $p=2$ leads:
    \begin{align}
        || \K\x - \y^{\delta_1} ||_2^2 \leq 2|| \K\x - \y^{\delta_2} ||_2^2 + 2 || \y^{\delta_1} - \y^{\delta_2} ||_2^2.
    \end{align}

    Similarly, considering $|| \w(\Psi(\y^{\delta_1}) \odot |\D \x| \> ||_1 = || \w(\Psi(\y^{\delta_2}) \odot |\D \x| + (\w(\Psi(\y^{\delta_1})) - \w(\Psi(\y^{\delta_2}))) \odot |\D \x| \> ||_1$ with $p=1$:
    \begin{align}
        || \w(\Psi(\y^{\delta_1}) \odot |\D \x| \> ||_1 \leq || \w(\Psi(\y^{\delta_2}) \odot |\D \x| \> ||_1 + ||(\w(\Psi(\y^{\delta_1})) - \w(\Psi(\y^{\delta_2}))) \odot |\D \x| \> ||_1.
    \end{align}

    Moreover, note that for any weights vector $\w\in\R^n$, 
    \begin{align}
        || \w \odot | \D \x | \> ||_1 = \sum_{i=1}^n  \w_i \left( |\D \x | \right)_i \leq \sum_{i=1}^n \w_i \sum_{i=1}^n \left( |\D \x | \right)_i = || \w ||_1 || \> | \D \x | \> ||_1,
    \end{align}
    where we used that both $\w_i$ and $\left( |\D \x | \right)_i$ are non-negative for any index $i$. Therefore:
    \begin{align}
    \begin{split}
        \mathcal{J}_{\Psi, \delta_1}(\x) &= || \K\x - \y^{\delta_1} ||_2^2 + \lambda || \w(\Psi(\y^{\delta_1})) \odot | \D \x | \> ||_1 \\ &\leq 2|| \K\x - \y^{\delta_2} ||_2^2 + 2 || \y^{\delta_1} - \y^{\delta_2} ||_2^2 + \lambda || \w(\Psi(\y^{\delta_2})) \odot |\D \x| \> ||_1 \\ & \hspace{14em}+ \lambda ||(\w(\Psi(\y^{\delta_1}) - \w(\Psi(\y^{\delta_2}))) \odot |\D \x| \> ||_1 \\& \leq 2|| \K\x - \y^{\delta_2} ||_2^2 + 2 \lambda || \w(\Psi(\y^{\delta_2})) \odot |\D \x| \> ||_1 + 2 || \y^{\delta_1} - \y^{\delta_2} ||_2^2 \\ & \hspace{14em} + \lambda ||(\w(\Psi(\y^{\delta_1}) - \w(\Psi(\y^{\delta_2})) \odot |\D \x| \> ||_1 \\& = 2\mathcal{J}_{\Psi, \delta_2}(\x) + 2 || \y^{\delta_1} - \y^{\delta_2} ||_2^2 + \lambda ||(\w(\Psi(\y^{\delta_1})) - \w(\Psi(\y^{\delta_2}))) \odot |\D \x| \> ||_1 \\ &\leq 2\mathcal{J}_{\Psi, \delta_2}(\x) + 2 || \y^{\delta_1} - \y^{\delta_2} ||_2^2 + \lambda ||\w(\Psi(\y^{\delta_1})) - \w(\Psi(\y^{\delta_2}))||_1 || \> |\D \x| \> ||_1,
    \end{split}
    \end{align}
    concluding the proof.
\end{proof}



\subsection{Proof of Lemma \ref{lemma:minimizers_bounded}}\label{app:minimizers_bounded}

\begin{proof}
    By contradiction, assume $\{ \x^*_{\Psi, \delta_k} \}_{k \in \mathbb{N}}$ is unbounded, i.e. $|| \x^*_{\Psi, \delta_k} ||_2 \to \infty$ as $k \to \infty$. Since $\mathcal{J}_{\Psi, \delta_k}(\x)$ is coercive as proved in Lemma \ref{lemma:coercivity}, then $\mathcal{J}_{\Psi, \delta_k} (\x^*_{\Psi, \delta_k}) \to \infty$ as $k \to \infty$. 

    Let $\bar{\x} \in \X$ such that $\bar{\x} \in \ker (\D)$. Clearly $\mathcal{R}_{\Psi, \delta_k}(\bar{\x}) = 0$ for any $k \in \mathbb{N}$. By the same reasoning used in the proof of Lemma \ref{lemma:delta_1_2_inequality}, it follows:
    \begin{align}
        \mathcal{J}_{\Psi, \delta_k}(\bar{x}) &= || \K \bar{\x} - \y^{\delta_k} ||_2^2 \leq 2|| \K \bar{\x} - \y^0 ||_2^2 + 2|| \y^{\delta_k} - \y^0 ||_2^2 \\&\leq 2|| \K \bar{\x} - \y^0 ||_2^2 + 2\delta_k^2 \to 2|| \K \bar{\x} - \y^0 ||_2^2 < \infty,
    \end{align}

    This implies that there exists a $\bar k\in \mathbb{N}$ such that $\mathcal{J}_{\Psi, \delta_k}(\bar{x}) \leq \mathcal{J}_{\Psi, \delta_k}(\x^*_{\Psi, \delta_k})$ for every $k>\bar k$, which is a contradiction since $\x^*_{\Psi, \delta_k}$ is a minimizer of $\mathcal{J}_{\Psi, \delta_k}(\x)$.
    
\end{proof}

\subsection{Proof of Lemma \ref{lemma:psi_1_2_inequality}}\label{app:psi_1_2_inequality}

\begin{proof}
    The proof is similar to Lemma \ref{lemma:delta_1_2_inequality}. Applying the inequality \ref{eq:norm_inequality} on $|| \w(\Psi_1(\y^{\delta}) \odot |\D \x| \> ||_1 = || \w(\Psi_2(\y^{\delta}) \odot |\D \x| + (\w(\Psi_1(\y^{\delta}) - \w(\Psi_2(\y^{\delta})) \odot |\D \x| \> ||_1$ with $p=1$, it follows:
    \begin{align}
        || \w(\Psi_1(\y^{\delta}) \odot |\D \x| \> ||_1 \leq || \w(\Psi_2(\y^{\delta}) \odot |\D \x| \> ||_1 + ||(\w(\Psi_1(\y^{\delta}) - \w(\Psi_2(\y^{\delta})) \odot |\D \x| \> ||_1.
    \end{align}

    Therefore:
    \begin{align}
    \begin{split}
        \mathcal{J}_{\Psi_1, \delta}(\x) &= || \K\x - \y^{\delta} ||_2^2 + \lambda || \w(\Psi_1(\y^{\delta})) \odot | \D \x | \> ||_1 \\ &\leq || \K\x - \y^{\delta} ||_2^2 + \lambda || \w(\Psi_2(\y^{\delta}) \odot |\D \x| \> ||_1 + \lambda ||\left(\w(\Psi_1(\y^{\delta})) - \w(\Psi_2(\y^{\delta}))\right) \odot |\D \x| \> ||_1\\& = \mathcal{J}_{\Psi_2, \delta}(\x) + \lambda ||\left(\w(\Psi_1(\y^{\delta})) - \w(\Psi_2(\y^{\delta}))\right) \odot |\D \x| \> ||_1 \\ &\leq \mathcal{J}_{\Psi, \delta_2}(\x) + \lambda ||\w(\Psi_1(\y^{\delta})) - \w(\Psi_2(\y^{\delta}))||_1 || \> |\D \x| \> ||_1,
    \end{split}
    \end{align}
    concluding the proof.
\end{proof}

The next lemma proves the boundedness of the sequence $\{ \x^*_{\Psi_k, \delta} \}_{k \in \mathbb{N}}$ of the minimizers of $\mathcal{J}_{\Psi_k, \delta}(\x)$.


\subsection{Proof of Lemma \ref{lemma:minimizers_bounded_2}}\label{app:minimizers_bounded_2}

\begin{proof}
    By contradiction, assume $\{ \x^*_{\Psi_k, \delta} \}_{k \in \mathbb{N}}$ is unbounded, i.e. $|| \x^*_{\Psi_k, \delta} ||_2 \to \infty$ as $k \to \infty$. Since $\mathcal{J}_{\Psi_k, \delta}(\x)$ is coercive as proved in Lemma \ref{lemma:coercivity}, then $\mathcal{J}_{\Psi_k, \delta} (\x^*_{\Psi_k, \delta}) \to \infty$ as $k \to \infty$. 

    Let $\bar{\x} \in \X$ such that $\bar{\x} \in \ker (\D)$. It follows:
    \begin{align}
        \mathcal{J}_{\Psi_k, \delta}(\bar{x}) &= || \K \bar{\x} - \y^{\delta} ||_2^2 < \infty,
    \end{align}

    which implies that $\mathcal{J}_{\Psi_k, \delta}(\bar{x}) \leq \mathcal{J}_{\Psi_k, \delta}(\x^*_{\Psi_k, \delta})$ for sufficiently large values of $k \in \mathbb{N}$, which is a contradiction as $\x^*_{\Psi_k, \delta}$ is a minimizer of $\mathcal{J}_{\Psi_k, \delta}(\x)$.
\end{proof}
\end{appendices}

\end{document}